\documentclass{amsart}
\usepackage{amssymb,amsmath,bbm}

\theoremstyle{plain}
\newtheorem{theorem}{Theorem}[section]
\newtheorem{proposition}[theorem]{Proposition}
\newtheorem{lemma}[theorem]{Lemma}
\newtheorem{corollary}[theorem]{Corollary}
\newtheorem{claim}[theorem]{Claim}

\theoremstyle{definition}
\newtheorem{definition}[theorem]{Definition}
\newtheorem{question}[theorem]{Question}

\theoremstyle{remark}
\newtheorem{remark}[theorem]{Remark}

\DeclareMathOperator{\var}{var}
\DeclareMathOperator{\Var}{Var}

\def\Xint#1{\mathchoice
{\XXint\displaystyle\textstyle{#1}} 
{\XXint\textstyle\scriptstyle{#1}} 
{\XXint\scriptstyle\scriptscriptstyle{#1}} 
{\XXint\scriptscriptstyle\scriptscriptstyle{#1}} 
\!\int}
\def\XXint#1#2#3{{\setbox0=\hbox{$#1{#2#3}{\int}$ }
\vcenter{\hbox{$#2#3$ }}\kern-.6\wd0}}

\def\dashint{\Xint-}

\begin{document}

\title[Variation of the maximal function]{On the variation of the Hardy-Littlewood maximal function}
\author{Ond\v{r}ej Kurka}
\thanks{The author is a junior researcher in the University Centre for Mathematical Modelling, Applied Analysis and Computational Mathematics (MathMAC)}
\address{Department of Mathematical Analysis, Faculty of Mathematics and Physics,
Charles University, Sokolovsk\'a 83, 186 75 Prague 8, Czech Republic}
\email{kurka.ondrej@seznam.cz}
\begin{abstract}
We show that a function $ f : \mathbb{R} \rightarrow \mathbb{R} $ of bounded variation satisfies
$$ \Var Mf \leq C \Var f $$
where $ Mf $ is the centered Hardy-Littlewood maximal function of $ f $. Consequently, the operator $ f \mapsto (Mf)' $ is bounded from $ W^{1,1}(\mathbb{R}) $ to $ L^{1}(\mathbb{R}) $. This answers a question of Haj\l asz and Onninen in the one-dimensional case.
\end{abstract}
\maketitle

\section{Introduction and main results}

The centered Hardy-Littlewood maximal function of $ f : \mathbbm{R}^{n} \rightarrow \mathbbm{R} $ is defined by
$$ Mf(x) = \sup _{r > 0} \; \dashint _{B(x,r)} |f(y)| \, dy. $$
J. Kinnunen proved in \cite{kinnunen} that the maximal operator $ f \mapsto Mf $ is bounded in the Sobolev space $ W^{1,p}(\mathbbm{R}^{n}) $ for $ 1 < p \leq \infty $ (see also \cite[Theorem 1]{hajonn}). Since then, regularity properties of maximal functions have been studied by many authors in various settings. J.~Kinnunen and P.~Lindqvist \cite{kinnlind} proved soon that the boundedness is fulfilled also by the local maximal operator. The regularity of the fractional maximal function was studied by J.~Kinnunen and E.~Saksman \cite{kinnsaks}. It was shown further by P.~Haj\l asz and J.~Onninen \cite{hajonn} that the local spherical maximal operator is bounded in $ W^{1,p}(\Omega ) $ when $ n/(n-1) < p < \infty $. And last but not least, H.~Luiro \cite{luiro1} generalized the original boundedness result and established the continuity of the centered maximal operator in $ W^{1,p}(\mathbbm{R}^{n}), 1 < p < \infty $. For other related results, see also e.g. \cite{aldperlaz2}, \cite{buckley}, \cite{hajliu1}, \cite{hajliu2}, \cite{korry}. For results considering other concepts than the weak differentiability, see \cite{hajmal}, \cite{luiro2}.

Kinnunen's method depends on the Hardy-Littlewood-Wiener theorem which is available only for $ p > 1 $. The case $ p = 1 $ turns out to be quite different and less approachable than the case $ p > 1 $. Because $ Mf \notin L^{1} $ whenever $ f $ is non-trivial, Kinnunen's result fails for $ p = 1 $. Still, one can ask whether the maximal function of $ f \in W^{1,1} $ belongs locally to $ W^{1,1} $. In~\cite{hajonn}, the authors posed the following question.

\begin{question}[Haj\l asz and Onninen] \label{quhajonn}
Is the operator $ f \mapsto |\nabla Mf| $ bounded from $ W^{1,1}(\mathbbm{R}^{n}) $ to $ L^{1}(\mathbbm{R}^{n}) $?
\end{question}

In the present work, we show that the answer is positive for $ n = 1 $. The question had been already answered positively in the non-centered one-dimensional case by H.~Tanaka \cite{tanaka}. This result was sharpened later by J.~M.~Aldaz and J.~P\'erez L\'azaro \cite{aldperlaz1} who proved that, for an arbitrary $ f : \mathbbm{R} \rightarrow \mathbbm{R} $ of bounded variation, its non-centered maximal function $ \widetilde{M} f $ is absolutely continuous and
$$ \Var \widetilde{M} f \leq \Var f. $$
We prove that such an inequality holds for the centered maximal function as well.

\begin{theorem} \label{thm}
Let $ f : \mathbbm{R} \rightarrow \mathbbm{R} $ be a function of bounded variation. Then
$$ \Var Mf \leq C \Var f $$
for a universal constant $ C $.
\end{theorem}

Question~\ref{quhajonn} and the validity of Theorem \ref{thm} were already studied in the discrete setting in \cite{bchp}. In the present paper, we do not care how small the constant $ C $ may be. It is a plausible hypothesis that the inequality holds for $ C = 1 $, in the same way as in the non-centered case (see also \cite[Question B]{bchp}).

Once Theorem \ref{thm} is proven, it is not difficult to derive the weak differentiability of $ Mf $. Note that $ Mf $ needs not to be continuous for an $ f $ of bounded variation, and so $ M $ does not possess such strong regularity properties as $ \widetilde{M} $. Anyway, for a weakly differentiable $ f $, everything is all right.

We prove two consequences of Theorem~\ref{thm} concerning the regularity of the maximal function. Although the first corollary is essentially an auxiliary result, it says something more than we need and one may found it interesting itself.

\begin{corollary} \label{corlocac}
Let $ f : \mathbbm{R} \rightarrow \mathbbm{R} $ be a measurable function with $ Mf \not\equiv \infty $. If $ f $ is locally absolutely continuous on an open set $ U $, then $ Mf $ is also locally absolutely continuous on $ U $.
\end{corollary}

The following consequence answers Question~\ref{quhajonn} in the one-dimensional case.

\begin{corollary} \label{corw11}
Let $ f \in W^{1,1}_{\mathrm{loc}}(\mathbbm{R}) $ be such that $ f' \in L^{1}(\mathbbm{R}) $. Then $ Mf $ is weakly differentiable and
$$ \Vert (Mf)' \Vert _{1} \leq C \Vert f' \Vert _{1} $$
for a universal constant $ C $.
\end{corollary}

We note that the above results hold for the local maximal function as well. In fact, the passage to the local maximal function makes no important difference, as discussed in Remark~\ref{remloc}.

We introduce here one more result which is a modified version of Lemma~\ref{keylemma}, the key ingredient of the proof of Theorem \ref{thm}. We expose the lemma here because we believe that it can be used for finding a solution of the more-dimensional Haj\l asz-Onninen problem.

\begin{lemma} \label{keylemmamod}
Let $ \varrho $ be a positive number. Let $ f : \mathbbm{R} \rightarrow \mathbbm{R} $ be a function of bounded variation and let $ \Lambda _{k}^{n}, n, k \in \mathbbm{Z}, $ be non-negative numbers. Assume that, for every $ (n,k) $ with $ \Lambda _{k}^{n} > 0 $, there are $ s < u < v < t $ such that
$$ (k - \varrho )2^{-n} \leq s, \quad t \leq (k + \varrho )2^{-n}, $$
$$ u - s \geq 2^{-n}, \quad v - u \geq 2^{-n}, \quad t - v \geq 2^{-n} $$
and
$$ \min \{ f(s), f(t) \} - \dashint _{u}^{v} f \geq \Lambda _{k}^{n}. $$
Then
$$ \sum _{n,k} \Lambda _{k}^{n} \leq C \Var f $$
where $ C $ depends only on $ \varrho $.
\end{lemma}

Actually, this lemma is proven only for $ \varrho \leq \frac{50}{4} $ but the version for general $ \varrho $ can be obtained by modifying of the constants in the proof.

\medskip

The paper is organized as follows. In Section~\ref{sec:propmaxfct}, we study relations between a function and its maximal function. In Section~\ref{sec:keylem}, we introduce the key lemma (Lemma~\ref{keylemma}), proof of which takes also Sections \ref{sec:dealA} and \ref{sec:dealB}. Finally, Theorem~\ref{thm} is proven in Section~\ref{sec:proofthm}. Its proof is based on two previous results which are Lemma~\ref{lemmsuvt} and Lemma~\ref{keylemma}. The paper concludes with Section~\ref{sec:proofcor}, devoted to the proof of Corollaries \ref{corlocac} and \ref{corw11}.

\medskip

In closing of this introductory section, we present some informal notes concerning the proof of Theorem~\ref{thm} which may be helpful but the reader may skip them as well.

We will study the variations of $ f $ and $ Mf $ using two simple structures. We introduce the structure for $ Mf $ first (Definition~\ref{defpeaketc}), as we want to show that $ f $ oscillates comparably with $ Mf $. This structure, called a \emph{peak}, consists of three points $ p < r < q $ such that the value of $ Mf $ at the middle point $ r $ is greater than the values at $ p $ and $ q $. The variation of $ Mf $ is related to the quantity
$$ Mf(r) - Mf(p) + Mf(r) - Mf(q), \leqno \textrm{(Q1)} $$
since the sum of these quantities for a suitable system of peaks almost realizes the variation of $ Mf $.

For a peak $ p < r < q $, there are two possibilities. If some values of $ f $ in the interval $ (p, q) $ are close to or greater than $ Mf(r) $, then the peak can be easily handled, as the variation of $ f $ over $ (p - \varepsilon , q + \varepsilon ) $ is close to or greater than the quantity (Q1). In the other case, the peak can not be handled so easily, and we call such a peak \emph{essential}. Our tool for working with essential peaks is Lemma~\ref{lemmsuvt} which allows us to pass to the second type of structures.

The structure for $ f $ is the point-interval-point system given by numbers $ s < u < v < t $. Similarly as above, the structure is endowed with a quantity which presents the impact on the variation of $ f $. This quantity is given by
$$ \min \{ f(s), f(t) \} - \dashint _{u}^{v} f. \leqno \textrm{(Q2)} $$
In contrast to the previous formula, values in the middle are expected to be \emph{less than} values on the boundary (and this is clearly not the only difference).

We will deal with a system of these structures. This will take a significant part of the paper, due to certain difficulties. First of all, the corresponding intervals $ [s, t] $ do not have to be disjoint. Even, a point can be an element of arbitrarily large number of intervals (unless the lengths of the intervals are comparable).

The aim of this part of the proof is to show that each structure has its own contribution to the variation of $ f $ given by (Q2). This aim is met by an abstract statement, provided in two versions. Lemma~\ref{keylemma} is the exact version, while Lemma~\ref{keylemmamod} above is the more elaborated version.

\section{A property of the maximal function} \label{sec:propmaxfct}

Throughout the whole proof of Theorem \ref{thm}, a function $ f : \mathbbm{R} \rightarrow \mathbbm{R} $ of bounded variation will be fixed. Without loss of generality, we will suppose that $ f \geq 0 $.

\begin{lemma} \label{lemm0}
Let $ r \in \mathbbm{R} $ and $ \omega > 0 $ be such that $ \dashint _{r-\omega }^{r+\omega } f = Mf(r) $. Let moreover $ p \in \mathbbm{R} $ satisfy $ r - \omega < p < r $ and $ Mf(p) \leq Mf(r) $. Then there is $ t \in (2p - (r-\omega), r + \omega ) $ such that
$$ f(t) \geq Mf(r) \quad \textrm{and} \quad f(t) \geq Mf(p) + \frac{Mf(r)-Mf(p)}{r-p} \cdot \omega . $$
\end{lemma}

\begin{proof}
We choose $ t \in (2p - (r-\omega), r + \omega ) $ so that $ f(t) \geq \dashint _{2p-(r-\omega )}^{r+\omega } f $. To show that the choice works, let us consider the interval $ (r - \omega , 2p - (r - \omega )) $ centered at $ p $. We have
$$ \dashint _{r-\omega }^{r+\omega } f = Mf(r) \quad \textrm{and} \quad \dashint _{r-\omega }^{2p-(r-\omega )} f \leq Mf(p) \leq Mf(r). $$
Immediately,
$$ f(t) \geq \dashint _{2p-(r-\omega )}^{r+\omega } f \geq Mf(r). $$
Further,
\begin{eqnarray*}
\int _{2p-(r-\omega )}^{r+\omega } f & = & \int _{r-\omega }^{r+\omega } f - \int _{r-\omega }^{2p-(r-\omega )} f \\
 & \geq & 2\omega \cdot Mf(r) - 2 \big( p-(r-\omega) \big) \cdot Mf(p) \phantom{\int _{r-\omega }^{r+\omega } f} \\
 & = & 2(r-p) \cdot Mf(p) + 2\omega \cdot \big( Mf(r)-Mf(p) \big) ,
\end{eqnarray*}
and so
$$ f(t) \geq \dashint _{2p-(r-\omega )}^{r+\omega } f \geq Mf(p) + \frac{Mf(r)-Mf(p)}{r-p} \cdot \omega . $$
\end{proof}

\begin{remark}
If $ p $ is regular in the sense that $ Mf(p) \geq f(p) $, then $ t $ fulfills
$$ \frac{f(t) - f(p)}{\omega } \geq \frac{Mf(r)-Mf(p)}{r-p}. $$
Note that $ \omega $ is close to $ t - p $ if $ p $ is close to $ r $. Thus, the average increase of $ f $ in $ (p,t) $ is comparable to the average increase of $ Mf $ in $ (p,r) $. We expected at first that this might lead to a simple proof of Theorem \ref{thm}, possibly with $ C = 1 $. Nevertheless, no simple proof was found at last.

In fact, by a modification of the proof of Lemma \ref{lemm0}, one can even find $ t \in (2p - (r-\omega), r + \omega ) $ such that
$$ \frac{f(t) - f(p)}{t-p} \geq \frac{Mf(r)-Mf(p)}{r-p}. $$
It is sufficient to choose $ t $ so that $ g(t) \geq \dashint _{2p-(r-\omega )}^{r+\omega } g $ for the function $ g(x) = f(x) - \frac{Mf(r)-Mf(p)}{r-p} \cdot x $.

Notice also that the last inequality from the proof gives
$$ \frac{Mf(p+\omega) - Mf(p)}{\omega } \geq \frac{Mf(r)-Mf(p)}{r-p}. $$
\end{remark}

\begin{definition} \label{defpeaketc}
\begin{itemize}
\item A \emph{peak} is the system consisting of three points $ p < r < q $ such that $ Mf(p) < Mf(r) $ and $ Mf(q) < Mf(r) $,
\item the \emph{variation} of a peak $ \mathbbm{p} = \{ p < r < q \} $ is given by
$$ \var \mathbbm{p} = Mf(r) - Mf(p) + Mf(r) - Mf(q), $$
\item the \emph{variation} of a system $ \mathbbm{P} $ of peaks is
$$ \var \mathbbm{P} = \sum _{\mathbbm{p} \in \mathbbm{P}} \var \mathbbm{p}, $$
\item a peak $ \mathbbm{p} = \{ p < r < q \} $ is \emph{essential} if $ \sup _{p < x < q} f(x) \leq Mf(r) - \frac{1}{4} \var \mathbbm{p} $,
\item for the top $ r $ of an essential peak $ p < r < q $, we define (see Lemma \ref{lemmomega})
$$ \omega (r) = \max \Big\{ \omega > 0 : \dashint _{r-\omega}^{r+\omega} f = Mf(r) \Big\}. $$
\end{itemize}
\end{definition}

\begin{lemma} \label{lemmomega}
Let $ \mathbbm{p} = \{ p < r < q \} $ be an essential peak. Then $ \omega (r) $ is well defined. Moreover,
$$ r - \omega (r) < p \quad \textrm{and} \quad q < r + \omega (r). $$
\end{lemma}

\begin{proof}
We have
$$ Mf(r) > \lim _{\omega \rightarrow \infty } \dashint _{r-\omega}^{r+\omega} f \quad \textrm{and} \quad Mf(r) > \lim _{\omega \searrow 0} \dashint _{r-\omega}^{r+\omega} f, $$
as $ Mf(r) > Mf(p) \geq \lim _{\omega \rightarrow \infty } \dashint _{p-\omega}^{p+\omega} f = \lim _{\omega \rightarrow \infty } \dashint _{r-\omega}^{r+\omega} f $ and
$$ (r-\omega , r+\omega ) \subset (p, q) \quad \Rightarrow \quad \dashint _{r-\omega}^{r+\omega} f \leq Mf(r) - \frac{1}{4} \var \mathbbm{p}. $$
It follows that $ \omega (r) $ is well defined. Moreover, at least one of the points $ p, q $ belongs to $ (r-\omega (r), r+\omega (r)) $. We may assume that $ p \in (r-\omega (r), r+\omega (r)) $. It remains to realize that also $ q \in (r-\omega (r), r+\omega (r)) $.

By Lemma \ref{lemm0}, we can find a $ t $ such that $ p < t < r + \omega (r) $ and $ f(t) \geq Mf(r) $. Since $ \mathbbm{p} $ is an essential peak, $ t $ is not an element of $ (p, q) $, and we obtain $ q \leq t < r + \omega (r) $.
\end{proof}

\begin{lemma} \label{lemmsuvt}
Let $ (x, y) $ be an interval of length $ L $. Let a non-empty system
$$ \mathbbm{P} = \Big\{ \mathbbm{p}_{i} = \{ p_{i} < r_{i} < q_{i} \} : 1 \leq i \leq m \Big\} $$
of essential peaks satisfy
$$ x \leq r_{1} < q_{1} \leq p_{2} < r_{2} < q_{2} \leq \dots \leq p_{m-1} < r_{m-1} < q_{m-1} \leq p_{m} < r_{m} \leq y $$
and
$$ 25L < \omega (r_{i}) \leq 50L, \quad 1 \leq i \leq m. $$
Then there are $ s < u < v < t $ such that
$$ x - 50L \leq s, \quad t \leq y + 50L, $$
$$ u - s \geq 4L, \quad v - u \geq L, \quad t - v \geq 4L $$
and
$$ \min \{ f(s), f(t) \} - \dashint _{u}^{v} f \geq \frac{1}{12} \var \mathbbm{P}. $$
\end{lemma}

\begin{proof}
We divide the proof into three parts. In parts I. and II., we consider two special cases and find appropriate numbers satisfying the improved inequality
$$ \min \{ f(s), f(t) \} - \dashint _{u}^{v} f \geq \frac{1}{4} \var \mathbbm{P}. \leqno (*) $$
The general case is considered in part III.

I. Let us assume that the system $ \mathbbm{P} $ consists of one peak $ \mathbbm{p} = \{ p < r < q \} $. First, we find $ s $ and $ t $ such that
$$ f(s) \geq Mf(r), \quad x - 50L \leq s \leq 2q - (r+\omega (r)), $$
$$ f(t) \geq Mf(r), \quad 2p - (r-\omega (r)) \leq t \leq y + 50L. $$
Due to the symmetry, it is sufficient to find a $ t $ only. Recall that $ r - \omega (r) < p $ by Lemma \ref{lemmomega}. Hence, a suitable $ t $ is given by Lemma \ref{lemm0}, since $ r + \omega (r) \leq y + 50L $.

We consider two possibilities.

(I.a) If $ q - p < 10L $, then we have
$$ s \leq 2q - (r+\omega (r)) < 2p + 20L - r - 25L < p - 5L < p - L/2 - 4L, $$
and it can be shown similarly that $ q + L/2 + 4L \leq t $. We take
$$ (u,v) = \left\{\begin{array}{ll} (p-L/2,p+L/2), & \quad Mf(p) \leq Mf(q), \\
(q-L/2,q+L/2), & \quad Mf(p) > Mf(q). \\
\end{array} \right. $$
We obtain
$$ \min \{ f(s), f(t) \} - \dashint _{u}^{v} f \geq Mf(r) - \min \{ Mf(p), Mf(q) \} \geq \frac{1}{2} \var \mathbbm{p}, $$
and $ (*) $ is proven.

(I.b) If $ q - p \geq 10L $, then we use
$$ s \leq 2q - (r+\omega (r)) < q, \quad p < 2p - (r-\omega (r)) \leq t $$ 
(here, Lemma \ref{lemmomega} is needed again). At the same time,
$$ \min \{ f(s), f(t) \} \geq Mf(r) > Mf(r) - \frac{1}{4} \var \mathbbm{p} \geq \sup _{p < x < q} f(x), $$
and so $ s $ and $ t $ can not belong to $ (p,q) $. It follows that
$$ s \leq p, \quad q \leq t. $$
Let us realize that the choice
$$ (u, v) = \big( (p + q - L)/2, (p + q + L)/2 \big) $$
works. Since $ u - p = (q - p - L)/2 = q - v $, we have
$$ u - s \geq u - p \geq 9L/2 , \quad t - v \geq q - v \geq 9L/2. $$
One can verify $ (*) $ by the computation 
$$ \min \{ f(s), f(t) \} - \dashint _{u}^{v} f \geq Mf(r) - \sup _{p < x < q} f(x) \geq \frac{1}{4} \var \mathbbm{p}. $$

II. Let us assume that the peaks are contained in the interval $ [x,y] $. (I.e., $ x \leq p_{1} $ and $ q_{m} \leq y $.) For $ 1 \leq i \leq m + 1 $, we define
$$ e_{i} = \left\{\begin{array}{ll} p_{i}, & \quad i = 1 \textrm{ or } Mf(p_{i}) \leq Mf(q_{i-1}), \\
q_{i-1}, & \quad i = m+1 \textrm{ or } Mf(p_{i}) > Mf(q_{i-1}). \\
\end{array} \right. $$
We work mainly with the modified system of peaks
$$ \widetilde{\mathbbm{P}} = \Big\{ \widetilde{\mathbbm{p}}_{i} = \{ e_{i} < r_{i} < e_{i+1} \} : 1 \leq i \leq m \Big\} . $$
For $ 1 \leq i \leq m $, let us find points $ s_{i} $ and $ t_{i} $ such that
$$ f(s_{i}) \geq Mf(r_{i}) \quad \textrm{and} \quad f(s_{i}) \geq Mf(e_{i+1}) + \frac{Mf(r_{i})-Mf(e_{i+1})}{e_{i+1}-r_{i}} \cdot \omega (r_{i}), $$
$$ x - 50L \leq s_{i} \leq x - 23L, $$
$$ f(t_{i}) \geq Mf(r_{i}) \quad \textrm{and} \quad f(t_{i}) \geq Mf(e_{i}) + \frac{Mf(r_{i})-Mf(e_{i})}{r_{i}-e_{i}} \cdot \omega (r_{i}), $$
$$ y + 23L \leq t_{i} \leq y + 50L. $$
Due to the symmetry, it is sufficient to find a $ t_{i} $ only. A suitable $ t_{i} $ is given by Lemma~\ref{lemm0}, since $ 2e_{i} - (r_{i}-\omega (r_{i})) \geq 2x - y + 25L = y + 23L $ and $ r_{i} + \omega (r_{i}) \leq y + 50L $.

Similarly as in part I., we consider two possibilities.

(II.a) Assume that
$$ \big| Mf(e_{m+1})-Mf(e_{1}) \big| > \frac{1}{2} \var \widetilde{\mathbbm{P}}. $$
We may assume moreover that $ Mf(e_{m+1}) > Mf(e_{1}) $. As $ \min \{ f(s_{m}), f(t_{m}) \} \geq Mf(r_{m}) > Mf(e_{m+1}) $, we obtain
$$ \min \{ f(s_{m}), f(t_{m}) \} - Mf(e_{1}) > Mf(e_{m+1}) - Mf(e_{1}) > \frac{1}{2} \var \widetilde{\mathbbm{P}} \geq \frac{1}{2} \var \mathbbm{P}. $$
The required properties including $ (*) $ are satisfied for
$$ s = s_{m}, \quad (u, v) = (e_{1}-L/2, e_{1}+L/2), \quad t = t_{m}. $$

(II.b) Assume that
$$ \big| Mf(e_{m+1})-Mf(e_{1}) \big| \leq \frac{1}{2} \var \widetilde{\mathbbm{P}}. $$
We have
$$ Mf(e_{m+1})-Mf(e_{1}) = \sum _{i=1}^{m} \Big[ \big( Mf(r_{i})-Mf(e_{i}) \big) - \big( Mf(r_{i})-Mf(e_{i+1}) \big) \Big] , $$
$$ \var \widetilde{\mathbbm{P}} = \sum _{i=1}^{m} \Big[ \big( Mf(r_{i})-Mf(e_{i}) \big) + \big( Mf(r_{i})-Mf(e_{i+1}) \big) \Big] , $$
and so the assumption can be written in the form
$$ \sum _{i=1}^{m} \big( Mf(r_{i})-Mf(e_{i}) \big) \geq \frac{1}{4} \var \widetilde{\mathbbm{P}} $$
$$ \quad \textrm{and} \quad \sum _{i=1}^{m} \big( Mf(r_{i})-Mf(e_{i+1}) \big) \geq \frac{1}{4} \var \widetilde{\mathbbm{P}}. $$
Let $ j $ and $ k $ be such that
$$ \frac{Mf(r_{j})-Mf(e_{j+1})}{e_{j+1}-r_{j}} = \max _{1 \leq i \leq m} \frac{Mf(r_{i})-Mf(e_{i+1})}{e_{i+1}-r_{i}}, $$
$$ \frac{Mf(r_{k})-Mf(e_{k})}{r_{k}-e_{k}} = \max _{1 \leq i \leq m} \frac{Mf(r_{i})-Mf(e_{i})}{r_{i}-e_{i}}. $$
We have
\begin{eqnarray*}
f(s_{j}) - Mf(e_{j+1}) & \geq & \frac{Mf(r_{j})-Mf(e_{j+1})}{e_{j+1}-r_{j}} \cdot \omega (r_{j}) \\
 & \geq & \frac{Mf(r_{j})-Mf(e_{j+1})}{e_{j+1}-r_{j}} \cdot 25L \\
 & \geq & \frac{Mf(r_{j})-Mf(e_{j+1})}{e_{j+1}-r_{j}} \cdot 25\sum _{i=1}^{m} (e_{i+1}-r_{i}) \\
 & = & 25\sum _{i=1}^{m} \frac{Mf(r_{j})-Mf(e_{j+1})}{e_{j+1}-r_{j}} \cdot (e_{i+1}-r_{i}) \\
 & \geq & 25\sum _{i=1}^{m} \big( Mf(r_{i})-Mf(e_{i+1}) \big) \\
 & \geq & \frac{25}{4} \var \widetilde{\mathbbm{P}}, \\
\end{eqnarray*}
and the same bound can be shown for $ f(t_{k}) - Mf(e_{k}) $. Hence,
$$ \min \{ f(s_{j}), f(t_{k}) \} - Mf(e) \geq \frac{25}{4} \var \widetilde{\mathbbm{P}} \geq \frac{25}{4} \var \mathbbm{P} $$
for some $ e \in \{ e_{j+1}, e_{k} \} $. The required properties including $ (*) $ are satisfied for
$$ s = s_{j}, \quad (u, v) = (e-L/2, e+L/2), \quad t = t_{k}. $$

III. In the general case, the system $ \mathbbm{P} $ can be divided into three subsystems
$$ \mathbbm{P}_{1} = \{ \mathbbm{p}_{i} : p_{i} < x \} , \quad \mathbbm{P}_{2} = \{ \mathbbm{p}_{i} : x \leq p_{i}, q_{i} \leq y \} , \quad \mathbbm{P}_{3} = \{ \mathbbm{p}_{i} : x \leq p_{i}, y < q_{i} \} . $$
Each of these systems consists of at most one peak or of peaks contained in $ [x,y] $. Thus, by parts I. and II. of the proof, if the system is non-empty, then there are appropriate numbers satisfying the improved inequality $ (*) $. The numbers $ s < u < v < t $ assigned to a $ \mathbbm{P}_{k} $ with $ \var \mathbbm{P}_{k} \geq \frac{1}{3} \var \mathbbm{P} $ work.
\end{proof}

\section{Key lemma} \label{sec:keylem}

In this section, we formulate our main tool for investigating the variation of the function $ f $. We introduce some notation concerning its proof but the main part of the proof will be accomplished in Sections~\ref{sec:dealA} and \ref{sec:dealB}.

\begin{lemma} \label{keylemma}
Let $ \Lambda _{k}^{n}, n \geq 0, k \in \mathbbm{Z}, $ be non-negative numbers such that only finitely of them are positive. Let $ L_{0} > 0 $ and $ L_{n} = 2^{-n}L_{0} $ for $ n \in \mathbbm{N} $. Assume that, for every $ (n,k) $ with $ \Lambda _{k}^{n} > 0 $, there are $ s < u < v < t $ such that
$$ (k-50)L_{n} \leq s, \quad t \leq (k+51)L_{n}, $$
$$ u - s \geq 4L_{n}, \quad v - u \geq L_{n}, \quad t - v \geq 4L_{n} $$
and
$$ \min \{ f(s), f(t) \} - \dashint _{u}^{v} f \geq \Lambda _{k}^{n}. $$
Then
$$ \sum _{n,k} \Lambda _{k}^{n} \leq 20000 \Var f. $$
\end{lemma}

We show at the end of this section how the lemma follows from the results of the next two sections. To finish the proof, it is just sufficient to apply Claim~\ref{claim10Vf} on every $ N $ and every $ K $.

It turns out that the systems obtained directly from the assumption of the lemma are not convenient for our purposes and an additional property is needed. In the following claim, we show that there are systems with one of two additional properties. Unfortunately, we will be able to handle only with one property at the same time, and this will mean twice as much work for us.

\begin{claim} \label{claimAB}
Let $ n \geq 0 $ and $ k \in \mathbbm{Z} $. If $ \Lambda _{k}^{n} > 0 $, then at least one of the following two conditions takes place:

{\rm (A)} There are $ s < \alpha < \beta < \gamma < \delta < t $ such that
$$ (k - 50)L_{n} \leq s, \quad t \leq (k + 51)L_{n}, $$
$$ \alpha - s \geq L_{n}, \quad \beta - \alpha \geq L_{n}, \quad \gamma - \beta = 2L_{n}, \quad \delta - \gamma \geq L_{n}, \quad t - \delta \geq L_{n} $$
and
$$ \min \{ f(s), f(t) \} - \max \Big\{ \dashint _{\alpha }^{\beta } f, \dashint _{\gamma }^{\delta } f \Big\} \geq \frac{1}{2} \Lambda _{k}^{n}. $$

{\rm (B)} There are $ \alpha < \beta < u < v < \gamma < \delta $ such that
$$ (k - 50)L_{n} \leq \alpha , \quad \delta \leq (k + 51)L_{n}, $$
$$ \beta - \alpha \geq L_{n}, \quad u - \beta \geq L_{n}, \quad v - u \geq L_{n}, \quad \gamma - v \geq L_{n}, \quad \delta - \gamma \geq L_{n} $$
and
$$ \min \Big\{ \dashint _{\alpha }^{\beta } f, \dashint _{\gamma }^{\delta } f \Big\} - \dashint _{u}^{v} f \geq \frac{1}{2} \Lambda _{k}^{n}. $$
\end{claim}

\begin{proof}
Let $ s < u < v < t $ be the points which the assumption of Lemma \ref{keylemma} gives for $ (n,k) $. We define
$$ \alpha = u - 3L_{n}, \quad \beta = u - 2L_{n}, \quad \gamma = v + 2L_{n}, \quad \delta = v + 3L_{n} $$
and look whether the inequality
$$ \min \Big\{ \dashint _{\alpha }^{\beta } f, \dashint _{\gamma }^{\delta } f \Big\} \geq \frac{1}{2} \Big( \min \{ f(s), f(t) \} + \dashint _{u}^{v} f \Big) $$
holds. If it holds, then (B) is satisfied. If it does not hold, then
$$ \dashint _{I} f \leq \frac{1}{2} \Big( \min \{ f(s), f(t) \} + \dashint _{u}^{v} f \Big) $$
where $ I $ is one of the intervals $ (\alpha , \beta ), (\gamma , \delta ) $. This inequality is fulfilled also for $ I = (u,v) $. Hence, (A) is satisfied for one of the choices
$$ \alpha ' = \alpha , \; \beta ' = \beta , \; \gamma ' = u, \; \delta ' = v, $$
$$ \alpha ' = u, \; \beta ' = v, \; \gamma ' = \gamma , \; \delta ' = \delta . $$
\end{proof}

\begin{definition}
We define
$$ \mathcal{A} = \big\{ (n,k) : \Lambda _{k}^{n} > 0 \textrm{ and (A) from Claim \ref{claimAB} is satisfied for } (n,k) \big\} , $$
$$ \mathcal{A}_{K}^{n} = \big\{ k \in \mathbbm{Z} : k = K \, \mathrm{mod} \, 200, (n,k) \in \mathcal{A} \big\} , \quad n \geq 0, \; 0 \leq K \leq 199, $$
$$ \mathcal{B} = \big\{ (n,k) : \Lambda _{k}^{n} > 0 \textrm{ and (B) from Claim \ref{claimAB} is satisfied for } (n,k) \big\} , $$
$$ \mathcal{B}_{K}^{n} = \big\{ k \in \mathbbm{Z} : k = K \, \mathrm{mod} \, 200, (n,k) \in \mathcal{B} \big\} , \quad n \geq 0, \; 0 \leq K \leq 199. $$
\end{definition}

\begin{claim} \label{claim10Vf}
For $ 0 \leq N \leq 9 $ and $ 0 \leq K \leq 199 $, we have
$$ \sum \Big\{ \Lambda _{k}^{n} : n = N \, \mathrm{mod} \, 10, k = K \, \mathrm{mod} \, 200 \Big\} \leq 10 \Var f. $$
\end{claim}

\begin{proof}
Using Claim \ref{claimAB} and Corollaries \ref{corA} and \ref{corB}, we can write
$$
\begin{aligned}
\sum \Big\{ \Lambda _{k}^{n} & : n = N \, \mathrm{mod} \, 10, k = K \, \mathrm{mod} \, 200 \Big\} \\
 & = \sum \Big\{ \Lambda _{k}^{n} : n = N \, \mathrm{mod} \, 10, k = K \, \mathrm{mod} \, 200 \textrm{ and } \Lambda _{k}^{n} > 0 \Big\} \\
 & = \sum \Big\{ \Lambda _{k}^{n} : n = N \, \mathrm{mod} \, 10, k \in \mathcal{A}_{K}^{n} \cup \mathcal{B}_{K}^{n} \Big\} \\
 & \leq 5 \Var f + 5 \Var f.
\end{aligned}
$$
\end{proof}

\section{Dealing with group $ \mathcal{A} $} \label{sec:dealA}

\begin{proposition} \label{propA}
Let $ 0 \leq N \leq 9 $ and $ 0 \leq K \leq 199 $. Let $ \eta \in \mathbbm{N} \cup \{ 0 \} $ and let $ n = 10\eta + N $. Then there is a system
$$ x_{1} < u_{1} < v_{1} < x_{2} < u_{2} < v_{2} < \dots < x_{m} < u_{m} < v_{m} < x_{m+1} $$
such that
$$ u_{1} - x_{1} \geq L_{n}, \quad v_{1} - u_{1} \geq L_{n}, \quad x_{2} - v_{1} \geq L_{n}, \quad \dots $$
and
$$ \sum _{i=1}^{m} \Big[ f(x_{i}) + f(x_{i+1}) - 2\dashint _{u_{i}}^{v_{i}} f \Big] \geq \frac{1}{5} \sum \Big\{ \Lambda _{k}^{o} : o = N \, \mathrm{mod} \, 10, o \leq n, k \in \mathcal{A}_{K}^{o} \Big\} . $$
\end{proposition}

To prove the proposition, we provide a method how to construct such a system for $ \eta $ when a system for $ \eta - 1 $ is already constructed. We suppose that there is a system
$$ X_{1} < U_{1} < V_{1} < X_{2} < U_{2} < V_{2} < \dots < X_{M} < U_{M} < V_{M} < X_{M+1} $$
such that
$$ U_{1} - X_{1} \geq 1024L_{n}, \quad V_{1} - U_{1} \geq 1024L_{n}, \quad X_{2} - V_{1} \geq 1024L_{n}, \quad \dots $$
and
$$ \sum _{I=1}^{M} \Big[ f(X_{I}) + f(X_{I+1}) - 2\dashint _{U_{I}}^{V_{I}} f \Big] \geq \frac{1}{5} \sum \Big\{ \Lambda _{k}^{o} : o = N \, \mathrm{mod} \, 10, o \leq n - 10, k \in \mathcal{A}_{K}^{o} \Big\} $$
(for $ \eta - 1 = -1 $, we may consider $ M = 0 $ and $ X_{1} = $ anything). We want to construct a system
$$ x_{1} < u_{1} < v_{1} < x_{2} < u_{2} < v_{2} < \dots < x_{m} < u_{m} < v_{m} < x_{m+1} $$
such that
$$ u_{1} - x_{1} \geq L_{n}, \quad v_{1} - u_{1} \geq L_{n}, \quad x_{2} - v_{1} \geq L_{n}, \quad \dots $$
and
$$ \sum _{i=1}^{m} \Big[ f(x_{i}) + f(x_{i+1}) - 2\dashint _{u_{i}}^{v_{i}} f \Big] \geq \sum _{I=1}^{M} \Big[ f(X_{I}) + f(X_{I+1}) - 2\dashint _{U_{I}}^{V_{I}} f \Big] + \frac{1}{5} \sum _{k \in \mathcal{A}_{K}^{n}} \Lambda _{k}^{n}. $$

For every $ k \in \mathcal{A}_{K}^{n} $, let us consider such a system as in (A) from Claim \ref{claimAB}. If we put $ s_{k} = s, t_{k} = t $ and choose a $ (\alpha _{k}, \beta _{k}) \in \{ (\alpha , \beta ), (\gamma , \delta ) \} $, we obtain a system $ s_{k} < \alpha _{k} < \beta _{k} < t_{k} $ such that
$$ (k - 50)L_{n} \leq s_{k}, \quad t_{k} \leq (k + 51)L_{n}, $$
$$ \alpha _{k} - s_{k} \geq L_{n}, \quad \beta _{k} - \alpha _{k} \geq L_{n}, \quad t_{k} - \beta _{k} \geq L_{n} $$
and
$$ \min \{ f(s_{k}), f(t_{k}) \} - \dashint _{\alpha _{k}}^{\beta _{k}} f \geq \frac{1}{2} \Lambda _{k}^{n}. $$
We require from the choice of $ (\alpha _{k}, \beta _{k}) \in \{ (\alpha , \beta ), (\gamma , \delta ) \} $ that
$$ \mathrm{dist} \, (X_{I}, (\alpha _{k}, \beta _{k})) \geq L_{n}, \quad 1 \leq I \leq M+1. $$

For an interval $ (c,d) $ and a $ k \in \mathbbm{Z} $, we will denote
$$ (c,d) \perp k \quad \Leftrightarrow \quad \mathrm{dist} \, \big( (c,d), ((k - 50)L_{n}, (k + 51)L_{n}) \big) \geq L_{n}. $$

\begin{lemma} \label{lemmUV}
Let $ (U, V) $ be an interval of length greater than $ 210L_{n} $. Then there are a subinterval $ (U', V') $ and a $ k $ with $ k = K \, \mathrm{mod} \, 200 $ such that
\begin{itemize}
\item $ \dashint _{U'}^{V'} f \leq \dashint _{U}^{V} f $,
\item $ V' - U' \geq 5L_{n} $,
\item $ U' = (k-100)L_{n} $ or $ V' = (k+100)L_{n} $,
\item $ (k-105)L_{n} \leq U' $ and $ V' \leq (k+105)L_{n} $,
\item $ (U', V') \perp l $ for every $ l \neq k $ with $ l = K \, \mathrm{mod} \, 200 $.
\end{itemize}
Moreover, we can wish that $ \dashint _{U'}^{V'} f \geq \dashint _{U}^{V} f $ instead of the first property.
\end{lemma}

\begin{proof}
Let $ g $ and $ h $ be the uniquely determined integers with $ g = h = K \, \mathrm{mod} \, 200 $ such that
$$ (g-105)L_{n} \leq U < (g+95)L_{n} \quad \textrm{and} \quad (h-95)L_{n} < V \leq (h+105)L_{n}. $$
We have $ g < h $ due to the assumption $ V - U > 210L_{n} $. The system
$$ U < (g+100)L_{n} < (g+300)L_{n} < \dots < (h-100)L_{n} < V $$
is a partition of $ (U, V) $ into intervals of length greater than $ 5L_{n} $. We choose a part the average value of $ f $ over which is less or equal to the average value of $ f $ over $ (U, V) $. (Respectively, greater or equal to the average value of $ f $ over $ (U, V) $ if we want to prove the moreover statement.) Such a subinterval $ (U', V') $ and the appropriate $ k $ with $ g \leq k \leq h $ and $ k = K \, \mathrm{mod} \, 200 $ have the required properties.
\end{proof}

\begin{claim} \label{clAUV}
Let $ (U, V) $ be an interval of length greater than $ 210L_{n} $. Then at least one of the following conditions is fulfilled:

{\rm (i)} There is an interval $ (c, d) \subset (U, V) $ with $ d - c \geq L_{n} $ such that $ (c,d) \perp l $ for every $ l \in \mathcal{A}_{K}^{n} $ and
$$ - \dashint _{c}^{d} f \geq - \dashint _{U}^{V} f. $$

{\rm (ii)} There are an interval $ (c, d) \subset (U, V) $ with $ d - c \geq L_{n} $ and a $ k \in \mathcal{A}_{K}^{n} $ such that $ (c,d) \perp l $ for every $ l \in \mathcal{A}_{K}^{n} \setminus \{ k \} $ and
$$ - \dashint _{c}^{d} f \geq - \dashint _{U}^{V} f + \frac{1}{10} \Lambda _{k}^{n}. $$

{\rm (iii)} There are a system
$$ c < d < y < c' < d' $$
with $ (c, d') \subset (U - 1023L_{n}, V + 1023L_{n}) $ and
$$ d - c \geq L_{n}, \quad y - d \geq L_{n}, \quad c' - y \geq L_{n}, \quad d' - c' \geq L_{n} $$
and a $ k \in \mathcal{A}_{K}^{n} $ such that $ (c,d') \perp l $ for every $ l \in \mathcal{A}_{K}^{n} \setminus \{ k \} $ and
$$ f(y) - \dashint _{c}^{d} f - \dashint _{c'}^{d'} f \geq - \dashint _{U}^{V} f + \frac{1}{10} \Lambda _{k}^{n}. $$
\end{claim}

\begin{proof}
Let $ (U', V') $ and $ k $ be as in Lemma \ref{lemmUV}. If $ k \notin \mathcal{A}_{K}^{n} $, then (i) is fulfilled for $ (c,d) = (U', V') $. So, let us assume that $ k \in \mathcal{A}_{K}^{n} $ (and thus that we have $ s_{k} < \alpha _{k} < \beta _{k} < t_{k} $ for this $ k $).

Let us assume moreover that $ U' = (k-100)L_{n} $ (the procedure is similar when $ V' = (k+100)L_{n} $, see below). We put
$$ W = U' + \frac{1}{5}(V'-U'). $$
We have $ W = \frac{4}{5}U' + \frac{1}{5}V' \leq \frac{4}{5}(k-100)L_{n} + \frac{1}{5}(k+105)L_{n} = (k-59)L_{n} \leq s_{k} - 9L_{n} $ and $ \beta _{k} \leq t_{k} \leq (k + 51)L_{n} = U' + 151L_{n} $. In particular,
$$ s_{k} - W \geq L_{n} \quad \textrm{and} \quad \beta _{k} \leq V' + 1023L_{n}. $$
Further, we have
$$ \dashint _{U'}^{W} f \leq \dashint _{U'}^{V'} f + \frac{4}{5} \cdot \frac{1}{2} \Lambda _{k}^{n} \quad \textrm{or} \quad \dashint _{W}^{V'} f \leq \dashint _{U'}^{V'} f - \frac{1}{5} \cdot \frac{1}{2} \Lambda _{k}^{n}. $$
If the second inequality takes place, then (ii) is fulfilled for $ (c,d) = (W, V') $. If the first inequality takes place, then (iii) is fulfilled for
$$ (c,d) = (U', W), \quad y = s_{k}, \quad (c',d') = (\alpha _{k}, \beta _{k}). $$

So, the claim is proven under the assumption $ U' = (k-100)L_{n} $. The proof under the assumption $ V' = (k+100)L_{n} $ can be done in a similar way. If we denote
$$ W' = V' - \frac{1}{5}(V'-U'), $$
then one can show that (ii) is fulfilled for $ (c,d) = (U', W') $ or (iii) is fulfilled for
$$ (c,d) = (\alpha _{k}, \beta _{k}), \quad y = t_{k}, \quad (c',d') = (W',V'). $$
\end{proof}

\begin{claim} \label{clAS}
There is a subset $ \mathcal{S} \subset \mathcal{A}_{K}^{n} $ for which there exists a system
$$ y_{1} < c_{1} < d_{1} < y_{2} < c_{2} < d_{2} < \dots < y_{j} < c_{j} < d_{j} < y_{j+1} $$
such that
$$ c_{1} - y_{1} \geq L_{n}, \quad d_{1} - c_{1} \geq L_{n}, \quad y_{2} - d_{1} \geq L_{n}, \quad \dots , $$
\begin{eqnarray*}
l \in \mathcal{A}_{K}^{n} \setminus \mathcal{S} & \Rightarrow & (c_{i}, d_{i}) \perp l, \; 1 \leq i \leq j, \\
l \in \mathcal{A}_{K}^{n} \setminus \mathcal{S} & \Rightarrow & \mathrm{dist} \, (y_{i}, (\alpha _{l}, \beta _{l})) \geq L_{n}, \; 1 \leq i \leq j+1,
\end{eqnarray*}
and
$$ \sum _{i=1}^{j} \Big[ f(y_{i}) + f(y_{i+1}) - 2\dashint _{c_{i}}^{d_{i}} f \Big] \geq \sum _{I=1}^{M} \Big[ f(X_{I}) + f(X_{I+1}) - 2\dashint _{U_{I}}^{V_{I}} f \Big] + \frac{1}{5} \sum _{k \in \mathcal{S}} \Lambda _{k}^{n}. $$
\end{claim}

\begin{proof}
We apply Claim \ref{clAUV} on the intervals $ (U_{I}, V_{I}), 1 \leq I \leq M $. We write the inequalities from Claim \ref{clAUV} in a form more familiar for our purposes:
\begin{eqnarray*}
 & \textrm{(i)} & f(X_{I}) + f(X_{I+1}) - 2\dashint _{c}^{d} f \geq f(X_{I}) + f(X_{I+1}) - 2\dashint _{U_{I}}^{V_{I}} f, \\
 & \textrm{(ii)} & f(X_{I}) + f(X_{I+1}) - 2\dashint _{c}^{d} f \geq f(X_{I}) + f(X_{I+1}) - 2\dashint _{U_{I}}^{V_{I}} f + \frac{1}{5} \Lambda _{k}^{n}, \\
 & \textrm{(iii)} & \Big[ f(X_{I}) + f(y) - 2\dashint _{c}^{d} f \Big] + \Big[ f(y) + f(X_{I+1}) - 2\dashint _{c'}^{d'} f \Big] \\
 & & \quad \quad \geq f(X_{I}) + f(X_{I+1}) - 2\dashint _{U_{I}}^{V_{I}} f + \frac{1}{5} \Lambda _{k}^{n}.
\end{eqnarray*}
We define $ \mathcal{S} $ as the set of those $ k $'s which appeared in (ii) or (iii) for some $ I $. One can construct the desired system by inserting the systems which we obtained from Claim \ref{clAUV} between $ X_{I} $'s.
\end{proof}

To finish the proof of Proposition \ref{propA}, it remains to show that, if a proper subset $ \mathcal{S} \subset \mathcal{A}_{K}^{n} $ has such a system as in Claim \ref{clAS}, then $ \mathcal{S} \cup \{ k \} $ where $ k \in \mathcal{A}_{K}^{n} \setminus \mathcal{S} $ has also such a system.

So, let $ \mathcal{S} $ and
$$ y_{1} < c_{1} < d_{1} < y_{2} < c_{2} < d_{2} < \dots < y_{j} < c_{j} < d_{j} < y_{j+1} $$
be as in Claim \ref{clAS} and let $ k \in \mathcal{A}_{K}^{n} \setminus \mathcal{S} $. Let $ \iota $ be the index such that $ y_{\iota } $ belongs to the connected component of $ \mathbbm{R} \setminus \bigcup _{i=1}^{j} [c_{i},d_{i}] $ which covers $ ((k - 50)L_{n}, (k + 51)L_{n}) $. We intend to obtain the desired system for $ \mathcal{S} \cup \{ k \} $ by replacing $ y_{\iota } $ with
$$ y < \alpha _{k} < \beta _{k} < y' $$
where
$$ y = \left\{\begin{array}{ll} y_{\iota }, & \quad y_{\iota } \leq \alpha _{k} - L_{n} \textrm{ and } f(y_{\iota }) \geq f(s_{k}), \\
s_{k}, & \quad \textrm{otherwise}, \\
\end{array} \right. $$
$$ y' = \left\{\begin{array}{ll} y_{\iota }, & \quad y_{\iota } \geq \beta _{k} + L_{n} \textrm{ and } f(y_{\iota }) \geq f(t_{k}), \\
t_{k}, & \quad \textrm{otherwise}. \\
\end{array} \right. $$
For every $ l \neq k $ with $ l = K \, \mathrm{mod} \, 200 $, we have
$$ \mathrm{dist} \, \big( ((k - 50)L_{n}, (k + 51)L_{n}), ((l - 50)L_{n}, (l + 51)L_{n}) \big) \geq 99L_{n} \geq L_{n}, $$
and thus
\begin{eqnarray*}
l \in \mathcal{A}_{K}^{n} \setminus (\mathcal{S} \cup \{ k \} ) & \Rightarrow & (\alpha _{k}, \beta _{k}) \perp l, \\
l \in \mathcal{A}_{K}^{n} \setminus (\mathcal{S} \cup \{ k \} ) & \Rightarrow & \mathrm{dist} \, (y, (\alpha _{l}, \beta _{l})) \geq L_{n} \textrm{ and } \mathrm{dist} \, (y', (\alpha _{l}, \beta _{l})) \geq L_{n}.
\end{eqnarray*}

Let us prove the inequality for the modified system. We note that, if $ j \geq 1 $, then the left side of the inequality for the original system can be written in the form
$$ f(y_{1}) - 2\dashint _{c_{1}}^{d_{1}} f + 2f(y_{2}) - 2\dashint _{c_{2}}^{d_{2}} f + \dots + 2f(y_{j}) - 2\dashint _{c_{j}}^{d_{j}} f + f(y_{j+1}). $$
We need to show that the modification of the system increased this quantity at least by $ \frac{1}{5} \Lambda _{k}^{n} $. What we need to show is
\begin{flalign*}
 & \textrm{when $ 1 < \iota < j + 1 $:} & 2f(y) - 2\dashint _{\alpha _{k}}^{\beta _{k}} f + 2f(y') & \geq 2f(y_{\iota }) + \frac{1}{5} \Lambda _{k}^{n}, & \\
 & \textrm{when $ 1 = \iota < j + 1 $:} & f(y) - 2\dashint _{\alpha _{k}}^{\beta _{k}} f + 2f(y') & \geq f(y_{\iota }) + \frac{1}{5} \Lambda _{k}^{n}, & \\
 & \textrm{when $ 1 < \iota = j + 1 $:} & 2f(y) - 2\dashint _{\alpha _{k}}^{\beta _{k}} f + f(y') & \geq f(y_{\iota }) + \frac{1}{5} \Lambda _{k}^{n}, & \\
 & \textrm{when $ 1 = \iota = j + 1 $:} & f(y) - 2\dashint _{\alpha _{k}}^{\beta _{k}} f + f(y') & \geq \frac{1}{5} \Lambda _{k}^{n}. &
\end{flalign*}
These inequalities, even with $ 1 $ instead of $ \frac{1}{5} $, follow from
$$ \quad f(y) - \dashint _{\alpha _{k}}^{\beta _{k}} f \geq \frac{1}{2} \Lambda _{k}^{n}, \quad f(y') - \dashint _{\alpha _{k}}^{\beta _{k}} f \geq \frac{1}{2} \Lambda _{k}^{n}, $$
$$ f(y) \geq f(y_{\iota }) \quad \textrm{or} \quad f(y') \geq f(y_{\iota }) $$
($ f(y) \geq f(y_{\iota }) $ is implied by $ y_{\iota } \leq \alpha _{k} - L_{n} $ and $ f(y') \geq f(y_{\iota }) $ is implied by $ y_{\iota } \geq \beta _{k} + L_{n} $).

The proof of Proposition \ref{propA} is completed.

\begin{corollary} \label{corA}
For $ 0 \leq N \leq 9 $ and $ 0 \leq K \leq 199 $, we have
$$ \sum \Big\{ \Lambda _{k}^{n} : n = N \, \mathrm{mod} \, 10, k \in \mathcal{A}_{K}^{n} \Big\} \leq 5 \Var f. $$
\end{corollary}

\begin{proof}
Let $ \eta $ be large enough such that
$$ \mathcal{A}_{K}^{o} \neq \varnothing \quad \Rightarrow \quad o \leq n $$
where $ n = 10\eta + N $. Let
$$ x_{1} < u_{1} < v_{1} < x_{2} < u_{2} < v_{2} < \dots < x_{m} < u_{m} < v_{m} < x_{m+1} $$
be the system which Proposition \ref{propA} gives for $ N, K $ and $ \eta $. For $ 1 \leq i \leq m $, let $ w_{i} \in (u_{i},v_{i}) $ be chosen so that
$$ f(w_{i}) \leq \dashint _{u_{i}}^{v_{i}} f. $$
We compute
\begin{eqnarray*}
\Var f & \geq & \sum _{i=1}^{m} \Big[ |f(w_{i}) - f(x_{i})| + |f(x_{i+1}) - f(w_{i})| \Big] \\ 
 & \geq & \sum _{i=1}^{m} \Big[ f(x_{i}) - f(w_{i}) + f(x_{i+1}) - f(w_{i}) \Big] \\
 & \geq & \sum _{i=1}^{m} \Big[ f(x_{i}) + f(x_{i+1}) - 2\dashint _{u_{i}}^{v_{i}} f \Big] \\
 & \geq & \frac{1}{5} \sum \Big\{ \Lambda _{k}^{o} : o = N \, \mathrm{mod} \, 10, o \leq n, k \in \mathcal{A}_{K}^{o} \Big\} \\
 & = & \frac{1}{5} \sum \Big\{ \Lambda _{k}^{o} : o = N \, \mathrm{mod} \, 10, k \in \mathcal{A}_{K}^{o} \Big\} .
\end{eqnarray*}
\end{proof}

\section{Dealing with group $ \mathcal{B} $} \label{sec:dealB}

\begin{proposition} \label{propB}
Let $ 0 \leq N \leq 9 $ and $ 0 \leq K \leq 199 $. Let $ \eta \in \mathbbm{N} \cup \{ 0 \} $ and let $ n = 10\eta + N $. Then there is a system
$$ \varphi _{1} < \psi _{1} < s_{1} < t_{1} < \varphi _{2} < \psi _{2} < s_{2} < t_{2} < \dots < s_{m} < t_{m} < \varphi _{m+1} < \psi _{m+1} $$
such that
$$ \psi _{1} - \varphi _{1} \geq L_{n}, \quad s_{1} - \psi _{1} \geq L_{n}, \quad t_{1} - s_{1} \geq L_{n}, \quad \varphi _{2} - t_{1} \geq L_{n}, \quad \dots $$
and
$$ \sum _{i=1}^{m} \Big[ \dashint _{\varphi _{i}}^{\psi _{i}} f + \dashint _{\varphi _{i+1}}^{\psi _{i+1}} f - 2\dashint _{s_{i}}^{t_{i}} f \Big] \geq \frac{1}{5} \sum \Big\{ \Lambda _{k}^{o} : o = N \, \mathrm{mod} \, 10, o \leq n, k \in \mathcal{B}_{K}^{o} \Big\} . $$
\end{proposition}

To prove the proposition, we provide a method how to construct such a system for $ \eta $ when a system for $ \eta - 1 $ is already constructed. We suppose that there is a system
$$ \Phi _{1} < \Psi _{1} < S_{1} < T_{1} < \Phi _{2} < \Psi _{2} < S_{2} < T_{2} < \dots < S_{M} < T_{M} < \Phi _{M+1} < \Psi _{M+1} $$
such that
$$ \Psi _{1} - \Phi _{1} \geq 1024L_{n}, \quad \hspace{-4pt} S_{1} - \Psi _{1} \geq 1024L_{n}, \quad \hspace{-4pt} T_{1} - S_{1} \geq 1024L_{n}, \quad \hspace{-4pt} \Phi _{2} - T_{1} \geq 1024L_{n}, \quad \hspace{-4pt} \dots $$
and
$$ \sum _{I=1}^{M} \Big[ \dashint _{\Phi _{I}}^{\Psi _{I}} f + \dashint _{\Phi _{I+1}}^{\Psi _{I+1}} f - 2\dashint _{S_{I}}^{T_{I}} f \Big] \geq \frac{1}{5} \sum \Big\{ \Lambda _{k}^{o} : o = N \, \mathrm{mod} \, 10, o \leq n - 10, k \in \mathcal{B}_{K}^{o} \Big\} $$
(for $ \eta - 1 = -1 $, we may consider $ M = 0, \Phi _{1} = $ anything and $ \Psi _{1} = \Phi _{1} + 1024L_{n} $). We want to construct a system
$$ \varphi _{1} < \psi _{1} < s_{1} < t_{1} < \varphi _{2} < \psi _{2} < s_{2} < t_{2} < \dots < s_{m} < t_{m} < \varphi _{m+1} < \psi _{m+1} $$
such that
$$ \psi _{1} - \varphi _{1} \geq L_{n}, \quad s_{1} - \psi _{1} \geq L_{n}, \quad t_{1} - s_{1} \geq L_{n}, \quad \varphi _{2} - t_{1} \geq L_{n}, \quad \dots $$
and
$$ \sum _{i=1}^{m} \Big[ \dashint _{\varphi _{i}}^{\psi _{i}} f + \dashint _{\varphi _{i+1}}^{\psi _{i+1}} f - 2\dashint _{s_{i}}^{t_{i}} f \Big] \geq \sum _{I=1}^{M} \Big[ \dashint _{\Phi _{I}}^{\Psi _{I}} f + \dashint _{\Phi _{I+1}}^{\Psi _{I+1}} f - 2\dashint _{S_{I}}^{T_{I}} f \Big] + \frac{1}{5} \sum _{k \in \mathcal{B}_{K}^{n}} \Lambda _{k}^{n}. $$

For every $ k \in \mathcal{B}_{K}^{n} $, let us consider such a system as in (B) from Claim \ref{claimAB}. We obtain a system $ \alpha _{k} < \beta _{k} < u_{k} < v_{k} < \gamma _{k} < \delta _{k} $ such that
$$ (k - 50)L_{n} \leq \alpha _{k} , \quad \delta _{k} \leq (k + 51)L_{n}, $$
$$ \beta _{k} - \alpha _{k} \geq L_{n}, \quad u_{k} - \beta _{k} \geq L_{n}, \quad v_{k} - u_{k} \geq L_{n}, \quad \gamma _{k} - v_{k} \geq L_{n}, \quad \delta _{k} - \gamma _{k} \geq L_{n} $$
and
$$ \min \Big\{ \dashint _{\alpha _{k}}^{\beta _{k}} f, \dashint _{\gamma _{k}}^{\delta _{k}} f \Big\} - \dashint _{u_{k}}^{v_{k}} f \geq \frac{1}{2} \Lambda _{k}^{n}. $$

Again, for an interval $ (c,d) $ and a $ k \in \mathbbm{Z} $, we denote
$$ (c,d) \perp k \quad \Leftrightarrow \quad \mathrm{dist} \, \big( (c,d), ((k - 50)L_{n}, (k + 51)L_{n}) \big) \geq L_{n}. $$

\begin{claim} \label{clBST}
Let $ (S, T) $ be an interval of length greater than $ 210L_{n} $. Then at least one of the following conditions is fulfilled:

{\rm (i)} There is an interval $ (c, d) \subset (S, T) $ with $ d - c \geq L_{n} $ such that $ (c,d) \perp l $ for every $ l \in \mathcal{B}_{K}^{n} $ and
$$ - \dashint _{c}^{d} f \geq - \dashint _{S}^{T} f. $$

{\rm (ii)} There are an interval $ (c, d) \subset (S, T) $ with $ d - c \geq L_{n} $ and a $ k \in \mathcal{B}_{K}^{n} $ such that $ (c,d) \perp l $ for every $ l \in \mathcal{B}_{K}^{n} \setminus \{ k \} $ and
$$ - \dashint _{c}^{d} f \geq - \dashint _{S}^{T} f + \frac{1}{10} \Lambda _{k}^{n}. $$

{\rm (iii)} There are a system
$$ c < d < \mu < \nu < c' < d' $$
with $ (c, d') \subset (S - 500L_{n}, T + 500L_{n}) $ and
$$ d - c \geq L_{n}, \quad \mu - d \geq L_{n}, \quad \nu - \mu \geq L_{n}, \quad c' - \nu \geq L_{n}, \quad d' - c' \geq L_{n} $$
and a $ k \in \mathcal{B}_{K}^{n} $ such that $ (c,d') \perp l $ for every $ l \in \mathcal{B}_{K}^{n} \setminus \{ k \} $ and
$$ \dashint _{\mu }^{\nu } f - \dashint _{c}^{d} f - \dashint _{c'}^{d'} f \geq - \dashint _{S}^{T} f + \frac{1}{10} \Lambda _{k}^{n}. $$
\end{claim}

\begin{proof}
This can be proven in the same way as Claim \ref{clAUV}.
\end{proof}

The main difference between proofs of Propositions \ref{propA} and \ref{propB} is that we need one more analogy of Claim \ref{clAUV} because there are intervals $ (\Phi _{I}, \Psi _{I}) $ instead of points $ X_{I} $. Even, two versions of this analogy are provided. Both versions are written at once in the manner that the inequalities belonging to the second version are written in square brackets (this concerns also the proof of the claim).

\begin{claim} \label{clBPhiPsi}
Let $ (\Phi , \Psi ) $ be an interval of length greater than $ 210L_{n} $. Then at least one of the following conditions is fulfilled:

{\rm (i*)} There is an interval $ (\mu , \nu ) \subset (\Phi , \Psi ) $ with $ \nu - \mu \geq L_{n} $ such that $ (\mu , \nu ) \perp l $ for every $ l \in \mathcal{B}_{K}^{n} $ and
$$ \dashint _{\mu }^{\nu } f \geq \dashint _{\Phi }^{\Psi } f. $$

{\rm (ii*)} There are an interval $ (\mu , \nu ) \subset (\Phi , \Psi ) $ with $ \nu - \mu \geq L_{n} $ and a $ k \in \mathcal{B}_{K}^{n} $ such that $ (\mu , \nu ) \perp l $ for every $ l \in \mathcal{B}_{K}^{n} \setminus \{ k \} $ and
$$ \dashint _{\mu }^{\nu } f \geq \dashint _{\Phi }^{\Psi } f + \frac{1}{10} \Lambda _{k}^{n} \quad \Big[ \; \textrm{resp. } \dashint _{\mu }^{\nu } f \geq \dashint _{\Phi }^{\Psi } f + \frac{1}{5} \Lambda _{k}^{n} \; \Big] . $$

{\rm (iii*)} There are a system
$$ \mu < \nu < c < d < \mu ' < \nu ' $$
with $ (\mu , \nu ') \subset (\Phi - 500L_{n}, \Psi + 500L_{n}) $ and
$$ \nu - \mu \geq L_{n}, \quad c - \nu \geq L_{n}, \quad d - c \geq L_{n}, \quad \mu ' - d \geq L_{n}, \quad \nu ' - \mu ' \geq L_{n} $$
and a $ k \in \mathcal{B}_{K}^{n} $ such that $ (\mu , \nu ') \perp l $ for every $ l \in \mathcal{B}_{K}^{n} \setminus \{ k \} $ and
$$ \dashint _{\mu }^{\nu } f - \dashint _{c}^{d} f + \dashint _{\mu '}^{\nu '} f \geq \dashint _{\Phi }^{\Psi } f + \frac{1}{10} \Lambda _{k}^{n} $$
$$ \Big[ \; \textrm{resp. } \dashint _{\mu }^{\nu } f - 2\dashint _{c}^{d} f + 2\dashint _{\mu '}^{\nu '} f \geq \dashint _{\Phi }^{\Psi } f + \frac{1}{5} \Lambda _{k}^{n} $$
$$ \hspace{0.6cm} \textrm{and } \, 2\dashint _{\mu }^{\nu } f - 2\dashint _{c}^{d} f + \dashint _{\mu '}^{\nu '} f \geq \dashint _{\Phi }^{\Psi } f + \frac{1}{5} \Lambda _{k}^{n} \; \Big] . $$
\end{claim}

\begin{proof}
By Lemma \ref{lemmUV}, there are a subinterval $ (\Phi ', \Psi ') $ and a $ k $ with $ k = K \, \mathrm{mod} \, 200 $ such that
\begin{itemize}
\item $ \dashint _{\Phi '}^{\Psi '} f \geq \dashint _{\Phi }^{\Psi } f $,
\item $ \Psi ' - \Phi ' \geq 5L_{n} $,
\item $ \Phi ' = (k-100)L_{n} $ or $ \Psi ' = (k+100)L_{n} $,
\item $ (k-105)L_{n} \leq \Phi ' $ and $ \Psi ' \leq (k+105)L_{n} $,
\item $ (\Phi ', \Psi ') \perp l $ for every $ l \neq k $ with $ l = K \, \mathrm{mod} \, 200 $.
\end{itemize}
If $ k \notin \mathcal{B}_{K}^{n} $, then (i*) is fulfilled for $ (\mu , \nu ) = (\Phi ', \Psi ') $. So, let us assume that $ k \in \mathcal{B}_{K}^{n} $ (and thus that we have $ \alpha _{k} < \beta _{k} < u_{k} < v_{k} < \gamma _{k} < \delta _{k} $ for this $ k $).

We provide the proof under the assumption $ \Phi ' = (k-100)L_{n} $ only (the procedure is similar when $ \Psi ' = (k+100)L_{n} $). We put
$$ \Theta = \Phi ' + \frac{1}{5}(\Psi '-\Phi '). $$
We have $ \Theta = \frac{4}{5}\Phi ' + \frac{1}{5}\Psi ' \leq \frac{4}{5}(k-100)L_{n} + \frac{1}{5}(k+105)L_{n} = (k-59)L_{n} \leq \alpha _{k} - 9L_{n} \leq u_{k} - 9L_{n} $ and $ \delta _{k} \leq (k + 51)L_{n} = \Phi ' + 151L_{n} $. In particular,
$$ u_{k} - \Theta \geq L_{n} \quad \textrm{and} \quad \delta _{k} \leq \Psi ' + 500L_{n}. $$
Further, we have
$$ \dashint _{\Phi '}^{\Theta } f \geq \dashint _{\Phi '}^{\Psi '} f - \frac{4}{5} \cdot \frac{1}{2} \Lambda _{k}^{n} \quad \textrm{or} \quad \dashint _{\Theta }^{\Psi '} f \geq \dashint _{\Phi '}^{\Psi '} f + \frac{1}{5} \cdot \frac{1}{2} \Lambda _{k}^{n} $$
$$ \Big[ \; \textrm{resp.} \quad \dashint _{\Phi '}^{\Theta } f \geq \dashint _{\Phi '}^{\Psi '} f - \frac{8}{5} \cdot \frac{1}{2} \Lambda _{k}^{n} \quad \textrm{or} \quad \dashint _{\Theta }^{\Psi '} f \geq \dashint _{\Phi '}^{\Psi '} f + \frac{2}{5} \cdot \frac{1}{2} \Lambda _{k}^{n} \; \Big] . $$
If the second inequality takes place, then (ii*) is fulfilled for $ (\mu , \nu ) = (\Theta , \Psi ') $. If the first inequality takes place, then (iii*) is fulfilled for
$$ (\mu , \nu ) = \left\{\begin{array}{ll} (\Phi ',\Theta ), & \dashint _{\Phi '}^{\Theta } f \geq \dashint _{\alpha _{k}}^{\beta _{k}} f, \\
(\alpha _{k}, \beta _{k}), & \dashint _{\Phi '}^{\Theta } f < \dashint _{\alpha _{k}}^{\beta _{k}} f, \\
\end{array} \right.
\quad (c,d) = (u_{k},v_{k}), \quad (\mu ', \nu ') = (\gamma _{k}, \delta _{k}). $$
The inequalities in (iii*) follow from
$$ \dashint _{\mu }^{\nu } f \geq \dashint _{\Phi '}^{\Theta } f \geq \dashint _{\Phi }^{\Psi } f - \frac{4}{5} \cdot \frac{1}{2} \Lambda _{k}^{n} \quad \Big[ \; \textrm{resp. } \dots - \frac{8}{5} \cdot \frac{1}{2} \Lambda _{k}^{n} \; \Big] , $$
$$ \dashint _{\mu }^{\nu } f - \dashint _{c}^{d} f \geq \frac{1}{2} \Lambda _{k}^{n}, \quad \dashint _{\mu '}^{\nu '} f - \dashint _{c}^{d} f \geq \frac{1}{2} \Lambda _{k}^{n}. $$
\end{proof}

\begin{claim} \label{clBT}
There is a subset $ \mathcal{T} \subset \mathcal{B}_{K}^{n} $ for which there exists a system
$$ \mu _{1} < \nu _{1} < c_{1} < d_{1} < \mu _{2} < \nu _{2} < c_{2} < d_{2} < \dots < c_{j} < d_{j} < \mu _{j+1} < \nu _{j+1} $$
such that
$$ \nu _{1} - \mu _{1} \geq L_{n}, \quad c_{1} - \nu _{1} \geq L_{n}, \quad d_{1} - c_{1} \geq L_{n}, \quad \mu _{2} - d_{1} \geq L_{n}, \quad \dots , $$
\begin{eqnarray*}
l \in \mathcal{B}_{K}^{n} \setminus \mathcal{T} & \Rightarrow & (c_{i}, d_{i}) \perp l, \; 1 \leq i \leq j, \\
l \in \mathcal{B}_{K}^{n} \setminus \mathcal{T} & \Rightarrow & (\mu _{i}, \nu _{i}) \perp l, \; 1 \leq i \leq j+1,
\end{eqnarray*}
and
$$ \sum _{i=1}^{j} \Big[ \dashint _{\mu _{i}}^{\nu _{i}} f + \dashint _{\mu _{i+1}}^{\nu _{i+1}} f - 2\dashint _{c_{i}}^{d_{i}} f \Big] \geq \sum _{I=1}^{M} \Big[ \dashint _{\Phi _{I}}^{\Psi _{I}} f + \dashint _{\Phi _{I+1}}^{\Psi _{I+1}} f - 2\dashint _{S_{I}}^{T_{I}} f \Big] + \frac{1}{5} \sum _{k \in \mathcal{T}} \Lambda _{k}^{n}. $$
\end{claim}

We note that, if $ j \geq 1 $ and $ M \geq 1 $, then the inequality can be written in the form
$$
\begin{aligned}
 & \dashint _{\mu _{1}}^{\nu _{1}} f - 2\dashint _{c_{1}}^{d_{1}} f + 2\dashint _{\mu _{2}}^{\nu _{2}} f - 2\dashint _{c_{2}}^{d_{2}} f + \dots + 2\dashint _{\mu _{j}}^{\nu _{j}} f - 2\dashint _{c_{j}}^{d_{j}} f + \dashint _{\mu _{j+1}}^{\nu _{j+1}} f \\
 & \geq \dashint _{\Phi _{1}}^{\Psi _{1}} f - 2\dashint _{S_{1}}^{T_{1}} f + 2\dashint _{\Phi _{2}}^{\Psi _{2}} f - \dots - 2\dashint _{S_{M}}^{T_{M}} f + \dashint _{\Phi _{M+1}}^{\Psi _{M+1}} f + \frac{1}{5} \sum _{k \in \mathcal{T}} \Lambda _{k}^{n}.
\end{aligned}
$$

\begin{proof}
If $ M = 0 $, then we can put $ \mathcal{T} = \varnothing , j = 0 $ and find a suitable interval $ (\mu _{1}, \nu _{1}) $ of length $ L_{n} $. So, let us assume that $ M \geq 1 $.

We apply Claim \ref{clBST} on the intervals $ (S_{I}, T_{I}), 1 \leq I \leq M $, and Claim \ref{clBPhiPsi} on the intervals $ (\Phi _{I}, \Psi _{I}), 1 \leq I \leq M + 1, $ (the first version for $ 1 < I < M + 1 $, the second version for $ I = 1, I = M + 1 $). We write the inequalities from Claim \ref{clBST} in a form more familiar for our purposes:
\begin{eqnarray*}
 & \textrm{(i)} & - 2\dashint _{c}^{d} f \geq - 2\dashint _{S_{I}}^{T_{I}} f, \\
 & \textrm{(ii)} & - 2\dashint _{c}^{d} f \geq - 2\dashint _{S_{I}}^{T_{I}} f + \frac{1}{5} \Lambda _{k}^{n}, \\
 & \textrm{(iii)} & - 2\dashint _{c}^{d} f + 2\dashint _{\mu }^{\nu } f - 2\dashint _{c'}^{d'} f \geq - 2\dashint _{S_{I}}^{T_{I}} f + \frac{1}{5} \Lambda _{k}^{n}.
\end{eqnarray*}
Concerning the inequalities from Claim \ref{clBPhiPsi}, we moreover specify which inequality will be applied for $ I $:
\begin{flalign*}
 & \textrm{(i*)} & & 1 < I < M + 1: & & 2\dashint _{\mu }^{\nu } f \geq 2\dashint _{\Phi _{I}}^{\Psi _{I}} f, & \\
 & & & I = 1 \textrm{ or } I = M + 1: & & \dashint _{\mu }^{\nu } f \geq \dashint _{\Phi _{I}}^{\Psi _{I}} f, & \\
 & \textrm{(ii*)} & & 1 < I < M + 1: & & 2\dashint _{\mu }^{\nu } f \geq 2\dashint _{\Phi _{I}}^{\Psi _{I}} f + \frac{1}{5} \Lambda _{k}^{n}, & \\
 & & & I = 1 \textrm{ or } I = M + 1: & & \dashint _{\mu }^{\nu } f \geq \dashint _{\Phi _{I}}^{\Psi _{I}} f + \frac{1}{5} \Lambda _{k}^{n}, & \\
 & \textrm{(iii*)} & & 1 < I < M + 1: & & 2\dashint _{\mu }^{\nu } f - 2\dashint _{c}^{d} f + 2\dashint _{\mu '}^{\nu '} f \geq 2\dashint _{\Phi _{I}}^{\Psi _{I}} f + \frac{1}{5} \Lambda _{k}^{n}, & \\
 & & & I = 1: & & \dashint _{\mu }^{\nu } f - 2\dashint _{c}^{d} f + 2\dashint _{\mu '}^{\nu '} f \geq \dashint _{\Phi _{I}}^{\Psi _{I}} f + \frac{1}{5} \Lambda _{k}^{n}, & \\
 & & & I = M + 1: & & 2\dashint _{\mu }^{\nu } f - 2\dashint _{c}^{d} f + \dashint _{\mu '}^{\nu '} f \geq \dashint _{\Phi _{I}}^{\Psi _{I}} f + \frac{1}{5} \Lambda _{k}^{n}. &
\end{flalign*}
We define $ \mathcal{T} $ as the set of those $ k $'s which appeared in (ii), (iii), (ii*) or (iii*) for some $ I $. One can construct the desired system by collecting the systems which we obtained from Claims \ref{clBST} and \ref{clBPhiPsi}.
\end{proof}

To finish the proof of Proposition \ref{propB}, it remains to show that, if a proper subset $ \mathcal{T} \subset \mathcal{B}_{K}^{n} $ has such a system as in Claim \ref{clBT}, then $ \mathcal{T} \cup \{ k \} $ where $ k \in \mathcal{B}_{K}^{n} \setminus \mathcal{T} $ has also such a system.

So, let $ \mathcal{T} $ and
$$ \mu _{1} < \nu _{1} < c_{1} < d_{1} < \mu _{2} < \nu _{2} < c_{2} < d_{2} < \dots < c_{j} < d_{j} < \mu _{j+1} < \nu _{j+1} $$
be as in Claim \ref{clBT} and let $ k \in \mathcal{B}_{K}^{n} \setminus \mathcal{T} $. Let $ \iota $ be the index such that $ (\mu _{\iota }, \nu _{\iota }) $ is covered by the same connected component of $ \mathbbm{R} \setminus \bigcup _{i=1}^{j} [c_{i},d_{i}] $ as $ ((k - 50)L_{n}, (k + 51)L_{n}) $. We intend to obtain the desired system for $ \mathcal{T} \cup \{ k \} $ by replacing $ \mu _{\iota } < \nu _{\iota } $ with
$$ \mu < \nu < u_{k} < v_{k} < \mu ' < \nu ' $$
where
$$ (\mu , \nu ) = \left\{\begin{array}{ll} (\mu _{\iota }, \nu _{\iota }), & \quad \nu _{\iota } \leq u_{k} - L_{n} \textrm{ and } \dashint _{\mu _{\iota }}^{\nu _{\iota }} f \geq \dashint _{\alpha _{k}}^{\beta _{k}} f, \\
(\alpha _{k}, \beta _{k}), & \quad \textrm{otherwise}, \\
\end{array} \right. $$
$$ (\mu ', \nu ') = \left\{\begin{array}{ll} (\mu _{\iota }, \nu _{\iota }), & \quad \mu _{\iota } \geq v_{k} + L_{n} \textrm{ and } \dashint _{\mu _{\iota }}^{\nu _{\iota }} f \geq \dashint _{\gamma _{k}}^{\delta _{k}} f, \\
(\gamma _{k}, \delta _{k}), & \quad \textrm{otherwise}. \\
\end{array} \right. $$
For every $ l \neq k $ with $ l = K \, \mathrm{mod} \, 200 $, we have
$$ \mathrm{dist} \, \big( ((k - 50)L_{n}, (k + 51)L_{n}), ((l - 50)L_{n}, (l + 51)L_{n}) \big) \geq 99L_{n} \geq L_{n}, $$
and thus
\begin{eqnarray*}
l \in \mathcal{B}_{K}^{n} \setminus (\mathcal{T} \cup \{ k \} ) & \Rightarrow & (u_{k}, v_{k}) \perp l, \\
l \in \mathcal{B}_{K}^{n} \setminus (\mathcal{T} \cup \{ k \} ) & \Rightarrow & (\mu, \nu) \perp l \textrm{ and } (\mu ', \nu ') \perp l.
\end{eqnarray*}

Let us prove the inequality for the modified system. We need to show that the modification of the system increased the left side at least by $ \frac{1}{5} \Lambda _{k}^{n} $. What we need to show is
\begin{flalign*}
 & \textrm{when $ 1 < \iota < j + 1 $:} & 2\dashint _{\mu }^{\nu } f - 2\dashint _{u_{k}}^{v_{k}} f + 2\dashint _{\mu '}^{\nu '} f & \geq 2\dashint _{\mu _{\iota }}^{\nu _{\iota }} f + \frac{1}{5} \Lambda _{k}^{n}, & \\
 & \textrm{when $ 1 = \iota < j + 1 $:} & \dashint _{\mu }^{\nu } f - 2\dashint _{u_{k}}^{v_{k}} f + 2\dashint _{\mu '}^{\nu '} f & \geq \dashint _{\mu _{\iota }}^{\nu _{\iota }} f + \frac{1}{5} \Lambda _{k}^{n}, & \\
 & \textrm{when $ 1 < \iota = j + 1 $:} & 2\dashint _{\mu }^{\nu } f - 2\dashint _{u_{k}}^{v_{k}} f + \dashint _{\mu '}^{\nu '} f & \geq \dashint _{\mu _{\iota }}^{\nu _{\iota }} f + \frac{1}{5} \Lambda _{k}^{n}, & \\
 & \textrm{when $ 1 = \iota = j + 1 $:} & \dashint _{\mu }^{\nu } f - 2\dashint _{u_{k}}^{v_{k}} f + \dashint _{\mu '}^{\nu '} f & \geq \frac{1}{5} \Lambda _{k}^{n}. &
\end{flalign*}
These inequalities, even with $ 1 $ instead of $ \frac{1}{5} $, follow from
$$ \quad \dashint _{\mu }^{\nu } f - \dashint _{u_{k}}^{v_{k}} f \geq \frac{1}{2} \Lambda _{k}^{n}, \quad \dashint _{\mu '}^{\nu '} f - \dashint _{u_{k}}^{v_{k}} f \geq \frac{1}{2} \Lambda _{k}^{n}, $$
$$ \dashint _{\mu }^{\nu } f \geq \dashint _{\mu _{\iota }}^{\nu _{\iota }} f \quad \textrm{or} \quad \dashint _{\mu '}^{\nu '} f \geq \dashint _{\mu _{\iota }}^{\nu _{\iota }} f $$
($ \dashint _{\mu }^{\nu } f \geq \dashint _{\mu _{\iota }}^{\nu _{\iota }} f $ is implied by $ \nu _{\iota } \leq u_{k} - L_{n} $ and $ \dashint _{\mu '}^{\nu '} f \geq \dashint _{\mu _{\iota }}^{\nu _{\iota }} f $ is implied by $ \mu _{\iota } \geq v_{k} + L_{n} $).

The proof of Proposition \ref{propB} is completed.

\begin{corollary} \label{corB}
For $ 0 \leq N \leq 9 $ and $ 0 \leq K \leq 199 $, we have
$$ \sum \Big\{ \Lambda _{k}^{n} : n = N \, \mathrm{mod} \, 10, k \in \mathcal{B}_{K}^{n} \Big\} \leq 5 \Var f. $$
\end{corollary}

\begin{proof}
Let $ \eta $ be large enough such that
$$ \mathcal{B}_{K}^{o} \neq \varnothing \quad \Rightarrow \quad o \leq n $$
where $ n = 10\eta + N $. Let
$$ \varphi _{1} < \psi _{1} < s_{1} < t_{1} < \varphi _{2} < \psi _{2} < s_{2} < t_{2} < \dots < s_{m} < t_{m} < \varphi _{m+1} < \psi _{m+1} $$
be the system which Proposition \ref{propB} gives for $ N, K $ and $ \eta $. For $ 1 \leq i \leq m + 1 $, let $ \theta _{i} \in (\varphi _{i}, \psi _{i}) $ be chosen so that
$$ f(\theta _{i}) \geq \dashint _{\varphi _{i}}^{\psi _{i}} f. $$
For $ 1 \leq i \leq m $, let $ z_{i} \in (s_{i},t_{i}) $ be chosen so that
$$ f(z_{i}) \leq \dashint _{s_{i}}^{t_{i}} f. $$
We compute
\begin{eqnarray*}
\Var f & \geq & \sum _{i=1}^{m} \Big[ |f(z_{i}) - f(\theta _{i})| + |f(\theta _{i+1}) - f(z_{i})| \Big] \\ 
 & \geq & \sum _{i=1}^{m} \Big[ f(\theta _{i}) - f(z_{i}) + f(\theta _{i+1}) - f(z_{i}) \Big] \\
 & \geq & \sum _{i=1}^{m} \Big[ \dashint _{\varphi _{i}}^{\psi _{i}} f + \dashint _{\varphi _{i+1}}^{\psi _{i+1}} f - 2\dashint _{s_{i}}^{t_{i}} f \Big] \\
 & \geq & \frac{1}{5} \sum \Big\{ \Lambda _{k}^{o} : o = N \, \mathrm{mod} \, 10, o \leq n, k \in \mathcal{B}_{K}^{o} \Big\} \\
 & = & \frac{1}{5} \sum \Big\{ \Lambda _{k}^{o} : o = N \, \mathrm{mod} \, 10, k \in \mathcal{B}_{K}^{o} \Big\} .
\end{eqnarray*}
\end{proof}

\section{Proof of Theorem \ref{thm}} \label{sec:proofthm}

We are going to finish the proof of Theorem~\ref{thm}. Recall that Theorem~\ref{thm} is being proven for a fixed function $ f $ of bounded variation with $ f \geq 0 $. We introduce the remaining notation needed for proving the theorem first. Note that some notation was already introduced in Definition \ref{defpeaketc}.

We fix a system
$$ a_{1} < b_{1} < a_{2} < b_{2} < \dots < a_{\sigma } < b_{\sigma } < a_{\sigma + 1} $$
such that
$$ Mf(a_{i}) < Mf(b_{i}) \quad \textrm{and} \quad Mf(a_{i+1}) < Mf(b_{i}) $$
for $ 1 \leq i \leq \sigma $.

\begin{definition} \label{defPEL}
\begin{itemize}
\item The system $ \mathbbm{P} $ consists of all peaks $ \mathbbm{p}_{i} = \{ a_{i} < b_{i} < a_{i+1} \} $ where $ 1 \leq i \leq \sigma $,
\item the system $ \mathbbm{E} $ consists of all essential peaks from $ \mathbbm{P} $,
\item $ L_{0} $ is given by $ 50L_{0} = \max (\{ \omega (b_{i}) : \mathbbm{p}_{i} \in \mathbbm{E} \} \cup \{ 0 \} ) $,
\item $ L_{n} $ is given by $ L_{n} = 2^{-n}L_{0} $ for $ n \in \mathbbm{N} $,
\item the systems $ \mathbbm{E}_{k}^{n}, n \geq 0, k \in \mathbbm{Z}, $ are defined by
$$ \mathbbm{E}_{k}^{n} = \big\{ \mathbbm{p}_{i} \in \mathbbm{E} : 25L_{n} < \omega (b_{i}) \leq 50L_{n} , kL_{n} \leq b_{i} < (k+1)L_{n} \big\} . $$
\end{itemize}
\end{definition}

Our aim is to prove the inequality $ \var \mathbbm{P} \leq C \Var f $. While the proof for the non-essential peaks is easy, the proof for the essential peaks employs all the previously achieved results.

\begin{lemma} \label{lemmnotess}
We have
$$ \var (\mathbbm{P} \setminus \mathbbm{E}) \leq 2 \Var f. $$
\end{lemma}

\begin{proof}
For every $ \mathbbm{p}_{i} \in \mathbbm{P} \setminus \mathbbm{E} $, we choose $ x_{i} $ with $ a_{i} < x_{i} < a_{i+1} $ such that
$$ f(x_{i}) \geq Mf(b_{i}) - \frac{1}{4} \var \mathbbm{p}_{i}. $$
We take a small enough $ \varepsilon > 0 $ such that the intervals $ (a_{i} - \varepsilon , a_{i} + \varepsilon ), 1 \leq i \leq \sigma + 1, $ are pairwise disjoint and do not contain any $ x_{j} $. For $ 1 \leq i \leq \sigma + 1 $, we choose $ y_{i} \in (a_{i} - \varepsilon , a_{i} + \varepsilon ) $ so that
$$ f(y_{i}) \leq Mf(a_{i}). $$
For $ \mathbbm{p}_{i} \in \mathbbm{P} \setminus \mathbbm{E} $, we have
\begin{eqnarray*}
|f(x_{i}) - f(y_{i})| + |f(y_{i+1}) - f(x_{i})| \hspace{-5pt} & \geq & \hspace{-5pt} f(x_{i}) - f(y_{i}) + f(x_{i}) - f(y_{i+1}) \\
 & \geq & \hspace{-5pt} 2 \Big[ Mf(b_{i}) - \frac{1}{4} \var \mathbbm{p}_{i} \Big] - Mf(a_{i}) - Mf(a_{i+1}) \\
 & = & \hspace{-5pt} \frac{1}{2} \var \mathbbm{p}_{i},
\end{eqnarray*}
and the lemma follows.
\end{proof}

\begin{lemma} \label{lemmess}
We have
$$ \var \mathbbm{E} \leq 12 \cdot 20000 \Var f. $$
\end{lemma}

\begin{proof}
Let us put
$$ \Lambda _{k}^{n} = \frac{1}{12} \var \mathbbm{E}_{k}^{n}, \quad n \geq 0, \; k \in \mathbbm{Z}, $$
and pick an $ (n,k) $ with $ \Lambda _{k}^{n} > 0 $. Clearly, the system $ \mathbbm{E}_{k}^{n} $ is non-empty. Let us consider $ x = kL_{n} $ and $ y = (k+1)L_{n} $. Then Lemma \ref{lemmsuvt} applied on $ \mathbbm{E}_{k}^{n} $ provides a system $ s < u < v < t $ such that
$$ (k - 50)L_{n} \leq s, \quad t \leq (k + 51)L_{n}, $$
$$ u - s \geq 4L_{n}, \quad v - u \geq L_{n}, \quad t - v \geq 4L_{n} $$
and
$$ \min \{ f(s), f(t) \} - \dashint _{u}^{v} f \geq \frac{1}{12} \var \mathbbm{E}_{k}^{n} = \Lambda _{k}^{n}. $$
It follows that the assumption of Lemma \ref{keylemma} is satisfied. We can write
$$ \frac{1}{12} \var \mathbbm{E} = \frac{1}{12} \sum _{n,k} \var \mathbbm{E}_{k}^{n} = \sum _{n,k} \Lambda _{k}^{n} \leq 20000 \Var f. $$
\end{proof}

Once we have these bounds, the proof of Theorem \ref{thm} is easy. Nevertheless, we provide the final argument for completeness.

\begin{proof}[Proof of Theorem \ref{thm}]
Let $ x_{1} < x_{2} < \dots < x_{l} $ be given. We want to show that
$$ \sum _{j=1}^{l-1} \big| Mf(x_{j+1})-Mf(x_{j}) \big| \leq C \Var f. $$
After eliminating unnecessary points and possible repeating of the first and the last point, we obtain a system
$$ b_{0} \leq a_{1} < b_{1} < a_{2} < b_{2} < \dots < a_{\sigma } < b_{\sigma } < a_{\sigma + 1} \leq b_{\sigma + 1} $$
such that
$$ Mf(a_{i}) < Mf(b_{i}) \quad \textrm{and} \quad Mf(a_{i+1}) < Mf(b_{i}) $$
for $ 1 \leq i \leq \sigma $ and
$$ \sum _{i=0}^{\sigma } \big( Mf(b_{i}) - Mf(a_{i+1}) \big) + \sum _{i=1}^{\sigma + 1} \big( Mf(b_{i}) - Mf(a_{i}) \big) = \sum _{j=1}^{l-1} \big| Mf(x_{j+1})-Mf(x_{j}) \big| . $$
Considering the notation from Definition \ref{defPEL}, the left side of this equality can be written as
$$ Mf(b_{0}) - Mf(a_{1}) + Mf(b_{\sigma + 1}) - Mf(a_{\sigma + 1}) + \var \mathbbm{P}. $$
We have $ Mf(b_{0}) - Mf(a_{1}) \leq \sup f - \inf f \leq \Var f $. Similarly, $ Mf(b_{\sigma + 1}) - Mf(a_{\sigma + 1}) \leq \Var f $. It follows now from Lemma \ref{lemmnotess} and Lemma \ref{lemmess} that
$$ \sum _{j=1}^{l-1} \big| Mf(x_{j+1})-Mf(x_{j}) \big| \leq (1 + 1 + 2 + 12 \cdot 20000) \Var f, $$
and the proof of the theorem is completed!
\end{proof}

\begin{remark} \label{remloc}
The proof of Theorem \ref{thm} works also for the local Hardy-Littlewood maximal function. More precisely, if $ \Omega \subset \mathbbm{R} $ is open and $ d : \Omega \rightarrow (0, \infty ) $ is Lipschitz with the constant $ 1 $ such that $ d(x) \leq \mathrm{dist}(x, \mathbbm{R} \setminus \Omega ) $, then the function
$$ M_{\leq d} f(x) = \sup _{0 < \omega \leq d(x)} \; \dashint _{x-\omega}^{x+\omega} |f| $$
fulfills $ \Var _{\Omega } M_{\leq d} f \leq C \Var _{\Omega } f $. Here, by $ \Var _{\Omega } $ we mean $ \sum _{n} \Var _{I_{n}} $ where $ \Omega = \bigcup _{n} I_{n} $ is a decomposition of $ \Omega $ into open intervals. The inequality $ \Var _{I_{n}} M_{\leq d} f \leq C \Var _{I_{n}} f $ can be proven in the same way as Theorem \ref{thm}. It is sufficient just to modify appropriately the formula for $ \omega (r) $ in Definition \ref{defpeaketc}.

The version of Corollary \ref{corw11} for $ M_{\leq d} f $ can be proven as well. If $ f \in W^{1,1}_{\mathrm{loc}}(\Omega ) $ and $ f' \in L^{1}(\Omega ) $, then $ M_{\leq d} f $ is weakly differentiable and
$$ \Vert (M_{\leq d} f)' \Vert _{1, \Omega } \leq C \Vert f' \Vert _{1, \Omega }. $$
\end{remark}

\section{Proof of Corollaries \ref{corlocac} and \ref{corw11}} \label{sec:proofcor}

In this section, we follow methods from \cite{aldperlaz1} and \cite{tanaka}. We recall that a function $ f : A \subset \mathbbm{R} \rightarrow \mathbbm{R} $ is said to have \emph{Lusin's property~$ (N) $} (or is called an \emph{$ N $-function}) on $ A $ if, for every set $ N \subset A $ of measure zero, $ f(N) $ is also of measure zero. The well-known Banach-Zarecki theorem states that $ f : [a,b] \rightarrow \mathbbm{R} $ is absolutely continuous if and only if it is a continuous $ N $-function of bounded variation.

\begin{lemma} \label{lemmlipsch}
Let $ f : \mathbbm{R} \rightarrow \mathbbm{R} $ be a measurable function with $ Mf \not\equiv \infty $ and let $ r > 0 $. Then the function
$$ M_{\geq r} f(x) = \sup _{\omega \geq r} \; \dashint _{x-\omega}^{x+\omega} |f| $$
is locally Lipschitz. In particular, $ M_{\geq r} f $ is a continuous $ N $-function.
\end{lemma}

We prove a claim first.

\begin{claim} \label{cllipsch}
For $ x, y \in \mathbbm{R} $, we have
$$ M_{\geq r} f(y) \geq M_{\geq r} f(x) - \frac{M_{\geq r} f(x)}{r}|y-x|. $$
\end{claim}

\begin{proof}
Due to the symmetry, we may assume that $ y > x $. Let $ \varepsilon > 0 $. There is an $ \omega \geq r $ for which
$$ \dashint _{x-\omega}^{x+\omega} |f| \geq M_{\geq r} f(x) - \varepsilon . $$
We can compute
\begin{eqnarray*}
M_{\geq r} f(y) \hspace{-6pt} & \geq \hspace{-6pt} & \dashint _{x-\omega}^{2y -(x-\omega)} |f| = \frac{1}{2(y-x+\omega)} \int _{x-\omega}^{2y -(x-\omega)} |f| \\
 & \geq \hspace{-6pt} & \frac{1}{2(y-x+\omega)} \int _{x-\omega}^{x+\omega} |f| \geq \frac{2\omega}{2(y-x+\omega)} \big( M_{\geq r} f(x) - \varepsilon \big) \\
 & \geq \hspace{-6pt} & M_{\geq r} f(x) - \varepsilon - \frac{y-x}{y-x+\omega} M_{\geq r} f(x) \geq M_{\geq r} f(x) - \varepsilon - \frac{y-x}{r} M_{\geq r} f(x).
\end{eqnarray*}
As $ \varepsilon > 0 $ could be chosen arbitrarily, the claim is proven.
\end{proof}

\begin{proof}[Proof of Lemma \ref{lemmlipsch}]
We realize first that $ M_{\geq r} f $ is locally bounded. If $ y \in \mathbbm{R} $, then $ M_{\geq r} f $ is bounded on a neighbourhood of $ y $ by Claim \ref{cllipsch}, as
$$ |y-x| < r \quad \Rightarrow \quad M_{\geq r} f(x) \leq \frac{r}{r-|y-x|} M_{\geq r} f(y). $$

Now, let $ I $ be a bounded interval. There is a $ B > 0 $ such that $ M_{\geq r} f(x) \leq B $ for $ x \in I $. Using Claim \ref{cllipsch} again, we obtain, for $ x, y \in I $,
$$ M_{\geq r} f(y) \geq M_{\geq r} f(x) - \frac{B}{r}|y-x|. $$
Hence, $ M_{\geq r} f $ is Lipschitz with the constant $ B/r $ on $ I $.
\end{proof}

\begin{lemma} \label{lemmcont}
Let $ f : \mathbbm{R} \rightarrow \mathbbm{R} $ be a measurable function with $ Mf \not\equiv \infty $ and let $ x \in \mathbbm{R} $. If $ f $ is continuous at $ x $, then $ Mf $ is continuous at $ x $, too.
\end{lemma}

\begin{proof}
The assumption $ Mf \not\equiv \infty $ is sufficient for $ Mf $ to be lower semicontinuous. Assume that $ Mf $ is not upper semicontinuous at $ x $. There is a sequence $ x_{k} $ converging to $ x $ such that $ \inf _{k \in \mathbbm{N}} Mf(x_{k}) > Mf(x) $. We choose $ c $ so that
$$ \inf _{k \in \mathbbm{N}} Mf(x_{k}) > c > Mf(x). $$
For each $ k \in \mathbbm{N} $, we choose $ \omega _{k} > 0 $ such that
$$ \dashint _{x_{k}-\omega_{k}}^{x_{k}+\omega_{k}} |f| \geq c, \quad k = 1, 2, \dots \; . $$
Now,
\begin{itemize}
\item the possibility $ \omega _{k} \rightarrow 0 $ contradicts the continuity of $ f $ at $ x $, since then $ \limsup _{y \rightarrow x} |f(y)| \geq c > Mf(x) \geq \liminf _{y \rightarrow x} |f(y)| $,
\item the possibility $ \limsup _{k \rightarrow \infty } \omega _{k} > r > 0 $ contradicts the continuity of the function $ M_{\geq r} f $ from Lemma \ref{lemmlipsch}, since then $ \limsup _{k \rightarrow \infty } M_{\geq r} f(x_{k}) \geq c > Mf(x) \geq M_{\geq r} f(x) $.
\end{itemize}
\end{proof}

\begin{lemma} \label{lemmluzin}
Let $ f : \mathbbm{R} \rightarrow \mathbbm{R} $ be a measurable function with $ Mf \not\equiv \infty $ which is continuous on an open set $ U $. If $ f $ has $ (N) $ on $ U $, then $ Mf $ has also $ (N) $ on $ U $.
\end{lemma}

\begin{proof}
Note that the set $ E = \{ x \in U : Mf(x) > |f(x)| \} $ fulfills $ E = \bigcup _{k=1}^{\infty } E_{1/k} $ where
$$ E_{r} = \Big\{ x \in U : Mf(x) > \sup _{|y-x|<r} |f(y)| \Big\} , \quad r > 0. $$
For $ x \in E_{r} $, we have $ Mf(x) = M_{\geq r} f(x) $ where $ M_{\geq r} f $ is as in Lemma \ref{lemmlipsch}. At the same time, for $ x \in U \setminus E $, we have $ Mf(x) = |f(x)| $. Hence,
\begin{eqnarray*}
|Mf(N)| & \leq & |Mf(N \setminus E)| + \sum _{k=1}^{\infty } |Mf(N \cap E_{1/k})| \\
 & \leq & |f(N \setminus E)| + \sum _{k=1}^{\infty } |M_{\geq 1/k}f(N \cap E_{1/k})| = 0
\end{eqnarray*}
for every null set $ N \subset U $.
\end{proof}

\begin{proof}[Proof of Corollary \ref{corlocac}]
By Lemma \ref{lemmcont}, $ Mf $ is continuous on $ U $. By Lemma \ref{lemmluzin}, $ Mf $ has $ (N) $ on $ U $. So, it is sufficient to show that $ Mf $ has bounded variation on a given $ [a,b] \subset U $ because then the Banach-Zarecki theorem can be applied to prove that $ Mf $ is absolutely continuous on $ [a,b] $.

Let $ r > 0 $ be chosen so that $ [a-r,b+r] \subset U $ and let $ g : \mathbbm{R} \rightarrow \mathbbm{R} $ be defined by
$$ g(x) = \left\{\begin{array}{ll} f(x), & \quad x \in [a-r,b+r], \\
0, & \quad x \notin [a-r,b+r]. \\
\end{array} \right. $$
Then $ g $ has bounded variation, as $ f $ is absolutely continuous on $ [a-r,b+r] $. By Theorem \ref{thm}, $ Mg $ has bounded variation. It remains to realize that
$$ Mf(x) = \max \{ Mg(x), M_{\geq r} f(x) \} , \quad x \in [a,b], $$
for the function $ M_{\geq r} f $ from Lemma \ref{lemmlipsch}.
\end{proof}

\begin{proof}[Proof of Corollary \ref{corw11}]
Assume that $ f \in W^{1,1}_{\mathrm{loc}}(\mathbbm{R}) $ is such that $ f' \in L^{1}(\mathbbm{R}) $. Then $ f $ is represented by a locally absolutely continuous function with variation $ \Vert f' \Vert _{1} $ (which will be also denoted by $ f $). By Corollary \ref{corlocac}, $ Mf $ is locally absolutely continuous, and thus weakly differentiable. Using Theorem \ref{thm}, we can write
$$ \Vert (Mf)' \Vert _{1} = \Var Mf \leq C \Var f = C \Vert f' \Vert _{1}. $$
\end{proof}

\section*{Acknowledgment}

The author is grateful to Jan Mal\'y for suggesting the problem.

\end{document}